\documentclass[english, a4paper, final]{article}

\usepackage[T1]{fontenc}
\usepackage{lmodern}

\usepackage{martin}
\usepackage[mathcal]{euscript}
\usepackage{multirow}
\usepackage{amssymb}
\usepackage{pifont}

\begin{document}
\title{Critical configurations for three projective views}
\author{Martin Bråtelund}
\date{\today}

\maketitle

\begin{abstract}
The problem of structure from motion is concerned with recovering the 3-dimensional structure of an object from a set of 2-dimensional images taken by unknown cameras. Generally, all information can be uniquely recovered if enough images and point correspondences are provided, yet there are certain cases where unique recovery is impossible; these are called \emph{critical configurations}. We use an algebraic approach to study the critical configurations for three projective cameras. We show that all critical configurations lie on the intersection of quadric surfaces, and classify exactly which intersections constitute a critical configuration. 
\end{abstract}



\textbf{NOTE:} The published version of this paper claims that a configuration of points lying on a twisted cubic passing through exactly one of the three camera centers is not critical. While this is generally true, the configuration will be critical if the line spanned by the two camera centers not lying on the twisted cubic is a secant to the twisted cubic. This version of the paper corrects this error.

\section{Introduction} 

In computer vision, one of the main problems is that of \emph{structure from motion}, where given a set of $2$-dimensional images the task is to reconstruct a scene in $3$-space and find the camera positions in this scene. Over time, many techniques have been developed for solving these problems for varying camera models and under different assumptions on the space and image points \cite{maybank1993theory, hartleyzisserman}. In general, with enough images and enough points in each image, one can uniquely recover the original scene. However, there are also some 3D-configurations of points and cameras where a unique recovery from the images is never possible. These are called \emph{critical configurations}.

Much work has been done to understand critical configurations for various settings \cite{buchanan1988twistedcubic, hartley2000, hartleyKahlAstrom2001, HKCalibrated, bertolini2007criticalOneView, HK, bertolini2020criticalRank, BunchananCriticalLines}, with results dating as far back as 1941 \cite{Krames1941}. While interesting from a purely theoretical viewpoint, critical configurations also play a part in practical applications. Even though critical configurations are rare in real-life reconstruction problems when enough data is available (due to noise), it has been shown that as the configurations approach the critical ones, reconstruction algorithms tend to become less and less stable \cite{stability, HK, bertoliniStability}.

\begin{figure}
\begin{center}
\includegraphics[width = 1\textwidth]{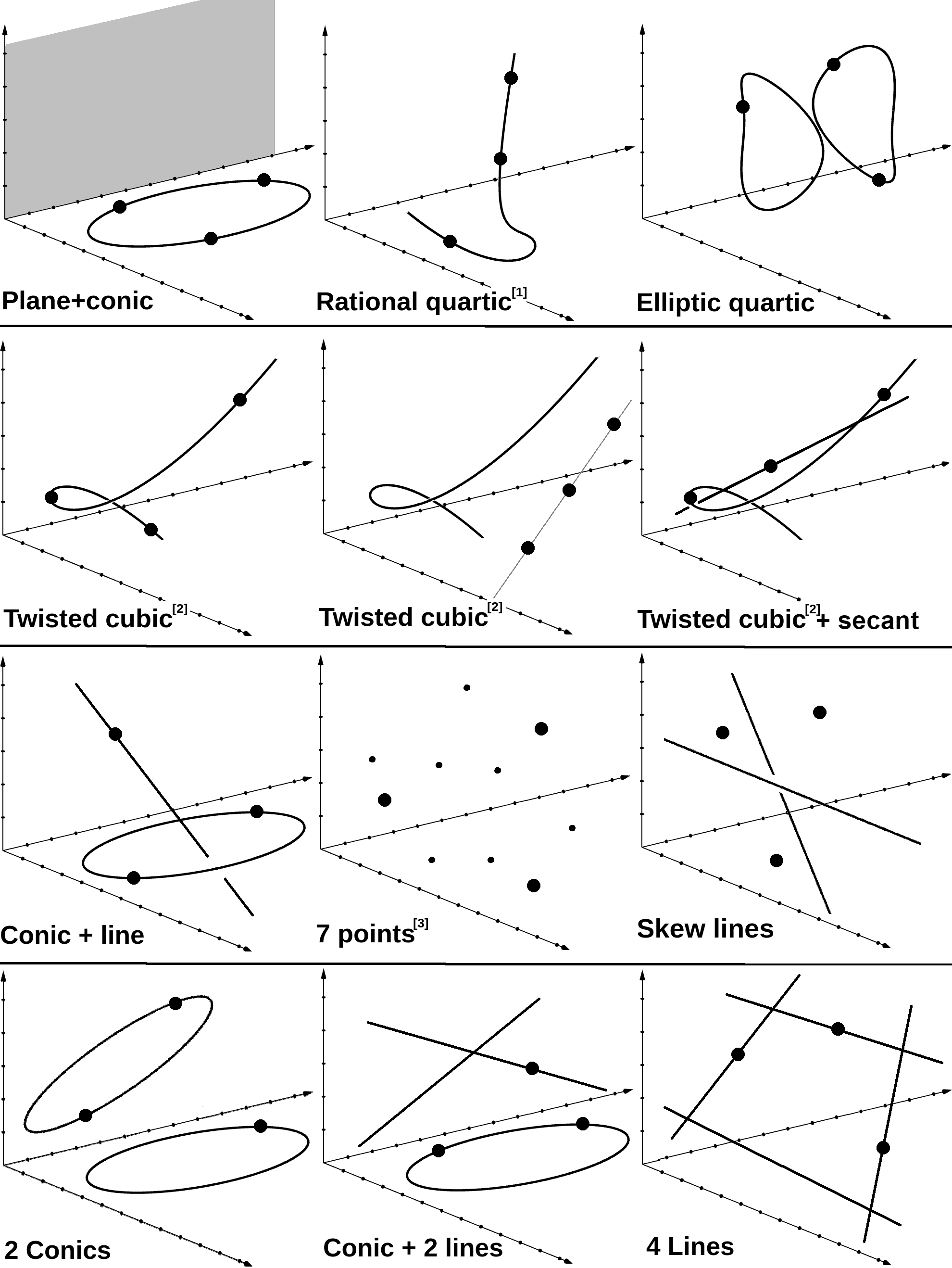}
\end{center}
\caption{The blow-down of some of the critical configurations for three views.\\
{[1]} The rational quartic may have a node or cusp.\\
{[2]} The twisted cubic may degenerate to the union of a conic and a line, or to three lines.\\
{[3]} The 7-point configuration must satisfy the conditions in \cref{prop:7_points_critical}.}
\label{fig_12critical}
\end{figure}

We will study the critical configurations for three projective cameras using a new approach. The idea is as follows (full details in \cref{sec:approach}): A set of $n$ cameras $\textbf{P}$ defines a rational map $\phi_{\textbf{P}}\from\p3\dashrightarrow(\p2)^n$, which extends to a morphism under blow-up of the camera centers:
\begin{equation*}
\begin{tikzcd}
\Bl_{\textbf{P}}(\p3) \arrow[r, "\widetilde{\phi_\textbf{P}}"] \arrow[d, "\pi_\textbf{P}"'] & (\p2)^{n} \\
\p3 \arrow[ru, "\phi_\textbf{P}"', dashed]                     &    
\end{tikzcd}
\end{equation*}
A set of $n$ cameras $\textbf{P}$ and a set of points $X\subset\Bl_{\textbf{P}}(\p3)$ is referred to as a \emph{configuration}. A configuration is critical if there exists another set of cameras $\textbf{Q}$ and points $Y$ such that $\widetilde{\phi_{\textbf{P}}}(X)=\widetilde{\phi_{\textbf{Q}}}(Y)$. The map $\widetilde{\phi_{\textbf{P}}}$ takes $\Bl_{\textbf{P}}(\p3)$ to a variety in $(\p2)^{n}$ called the \emph{multi-view variety}. Since $X$ and $Y$ both map into the intersection of the multi-view varieties of $\widetilde{\phi_{\textbf{P}}}$ and $\widetilde{\phi_{\textbf{Q}}}$, one can classify critical configurations by classifying all possible intersections of multi-view varieties.

We use this approach to recover many of the results in \cite{HK}. \cref{fig_12critical} shows the blow-down of some of the critical configurations for three views, that is, it shows the position of the camera centers (large dots), along with the blow-down $\pi_{\textbf{P}}(X)\subset\p3$ of the points (in most cases, a curve). The critical configurations not shown in \cref{fig_12critical} are subconfigurations of the ones shown. The main result of this paper is as follows:

\begin{theorem}
\label{thr:main_theorem}
A configuration of three projective cameras $\textbf{P}$ and 8 or more points $X\subset\Bl_{\textbf{P}}(\p3)$ form a critical configuration if and only if the points $X$ lie on the strict transform of one of the varieties described in \cref{fig_12critical}, or if they form a subset of such a variety. A configuration of three projective cameras and 6 or fewer points is always a critical configuration.
\end{theorem}

The case of 7 points and three projective cameras is covered in \cref{prop:7_points_critical}. 

The result in \cref{thr:main_theorem} is similar to what was found in \cite{HK}, although one critical configuration (twisted cubic + secant line) is new. Still, even though many of the results are the same, the techniques for obtaining them are new. We use the techniques developed in \cite{twoViews} and classify the critical configurations by considering the possible intersections of multi-view varieties, rather than relying on computations and ad hoc examples. We also hope that these new techniques will allow classification of critical configurations in more complicated scenarios in future work. 

In \cref{sec:background}, we introduce the main concepts, such as cameras and critical configurations, as well as some general results on critical configurations. \cref{sec:approach} describes the approach used to find critical configurations. \cref{sec:two-view_results} gives a summary of the key results from the two-view case necessary for the three-view classification. Sections \ref{sec:preliminary_three_views} and \ref{sec:critical_configurations_for_three_views} make up the main part of this paper. Here we show that critical configurations for three views need to lie on the intersection of three ruled quadrics and describe exactly which such intersections constitute a critical configuration. 

\section{Background}
\label{sec:background}
As this paper deals with a very similar topic as \cite{twoViews}, this section is more or less the same as Section 2 in \cite{twoViews}. We refer the reader to \cite{hartleyzisserman} for the basics on computer vision and multi-view geometry.

Let $\mathbb{C}$ denote the complex numbers, and let $\p{n}$ denote the projective space over the vector space $\mathbb{C}^{n+1}$. Projection from a point $p\in\p3$ is a linear map
\begin{equation*}
P'\from \p3\dashrightarrow\p2.
\end{equation*}
We refer to such a projection and its projection center $p$ as a \emph{camera} and its \emph{camera center} in $\p3$ (following established terminology, we use the words \emph{camera} and \emph{view} interchangeably). Following this theme, we refer to points in $\p3$ as \emph{space points} and points in $\p2$ as \emph{image points}. Similarly, $\p2$ will be referred to as an \emph{image}.

Once coordinates are chosen in $\p3$ and $\p2$, a camera can be represented by a $3\times4$ matrix of full rank called the \emph{camera matrix}. The camera center is then given as the kernel of the matrix. For the most part, we make no distinction between a camera and its camera matrix, referring to both simply as cameras. We use the \emph{real projective pinhole camera model}, meaning that we require a camera matrix to be of full rank and to have only real entries.

\begin{remark}
\label{rem:basis_chosen}
Throughout the paper, whenever we talk about cameras it is to be understood that a choice of coordinates has been made, both on the images and 3-space.
\end{remark}

Since the map $P'$ is not defined at the camera center $p$, it is not a morphism. This problem can be mended by taking the blow-up. Let $\Bl_{P}(\p3)$ be the blow-up of $\p3$ in the camera center of $P'$. We then get the following diagram:
\begin{equation*}
\begin{tikzcd}
\Bl_{P}(\p3) \arrow[r, "P"] \arrow[d, "\pi_P"'] & \p2 \\
\p3 \arrow[ru, "P'"', dashed]                     &    
\end{tikzcd}
\end{equation*}
where $\pi_P$ denotes the blow-down of $\Bl_{P}(\p3)$. This gives a morphism from $\Bl_{P}(\p3)$ to $\p2$. We denote the cameras by $P'$ when talking about rational maps and $P$ when talking about morphisms (see diagram above). For ease of notation, we retain the names \emph{camera} and \emph{camera center}, although one should note that in $\Bl_{P}(\p3)$ the camera center is no longer a point, but an exceptional divisor. 
\begin{definition}
A set of cameras $P_i$ is said to have \emph{collinear camera centers} if the camera centers of $P_i'$ lie on a line.
\end{definition}

\begin{remark}
Throughout the paper, when dealing with multiple cameras, we assume that all camera centers are distinct.
\end{remark}

\begin{definition}
Given an $n$-tuple of cameras $\textbf{P}=(P_1,\ldots,P_n)$, with camera centers $p_1,\ldots,p_n$, let $\Bl_{\textbf{P}}(\p3)$ denote the blow-up of $\p3$ in the $n$ camera centers and let $\pi_i$ denote the blow-down
\begin{align*}
\pi_i\from \Bl_{\textbf{P}}(\p3)\to \Bl_{P_i}(\p3).
\end{align*}
We define the \emph{joint camera map} to be the map
\begin{align*}
\widetilde{\phi_\textbf{P}}=P_1\times\cdots\times P_n\from\Bl_{\textbf{P}}(\p3)&\to(\p2)^{n},\\
x&\mapsto (P_1(\pi_1(x)),\ldots ,P_n(\pi_n(x))).
\end{align*}
\end{definition}

This again gives a commutative diagram
\begin{equation*}
\begin{tikzcd}
\Bl_{\textbf{P}}(\p3) \arrow[r, "\widetilde{\phi_\textbf{P}}"] \arrow[d, "\pi_\textbf{P}"'] & (\p2)^{n} \\
\p3 \arrow[ru, "\phi_\textbf{P}"', dashed]                     &    
\end{tikzcd}
\end{equation*}

\begin{remark}
The reason we use the blow-up $\Bl_{\textbf{P}}(\p3)$ rather than $\p3$ is to turn the cameras (and hence the joint camera map) into morphisms rather than rational maps. This ensures that the image of the joint camera map is Zariski closed, turning it into a projective variety. This in turn makes the maximal critical configurations (\cref{def:critical_configuration}) Zariski closed, simplifying the classification, otherwise, the maximal configurations would be constructible sets (see \cref{rem:quadric_minus_lines}).
\end{remark}

\begin{definition}
We denote the image of the joint camera map $\widetilde{\phi_\textbf{P}}$ as the \emph{multi-view variety} of $P_1,\ldots,P_n$. The set of all multi-homogeneous polynomials vanishing on $\im(\widetilde{\phi_\textbf{P}})$ is an ideal that we denote as the \emph{multi-view ideal}.
\end{definition}

\textbf{Notation.} Most works on computer vision do not use the blow-up, defining cameras/the joint camera map as rational maps rather than morphisms. Hence, they tend to define the multi-view variety as the \emph{closure} of the image rather than as the image itself. While our definition of the multi-view variety seems different, it is equivalent to that used in other works, like \cite{Tomas}.

\textbf{Notation.} While the multi-view variety is always irreducible, we use the term \emph{variety} to also include reducible algebraic sets.

\begin{theorem}[{\cite[Theorem 3.7]{Tomas}}]
\label{thr:ideal_of_multiview_variety}
The ideal of the multi-view variety is generated only by bilinear and trilinear forms. In particular, it is generated by the determinant of
\begin{align*}
\begin{bmatrix}
P_i'&\textbf{x}&\textbf{0}\\
P_j'&\textbf{0}&\textbf{y}
\end{bmatrix},
\end{align*}
for each pair of cameras, and the $7\times7$ minors of 
\begin{align*}
\begin{bmatrix}
P_i'&\textbf{x}&\textbf{0}&\textbf{0}\\
P_j'&\textbf{0}&\textbf{y}&\textbf{0}\\
P_k'&\textbf{0}&\textbf{0}&\textbf{z}
\end{bmatrix},
\end{align*}
for each triple. Here $\textbf{x}$, $\textbf{y}$ and $\textbf{z}$ are the $3\times1$ vectors with variables in the $i$-th, $j$-th and $k$-th image respectively.
\end{theorem}

\begin{definition}
Given a set of points $S\subset(\p2)^{n}$, a \emph{reconstruction} of $S$ is a configuration of cameras and points $(P_1,\ldots,P_n,X)$ such that $S=\widetilde{\phi_\textbf{P}}(X)$ where $\widetilde{\phi_\textbf{P}}$ is the joint camera map of the cameras $P_1,\ldots,P_n$. 
\end{definition}

\begin{definition}
Given a configuration of cameras and points $(P_1,\ldots,P_n,X)$, we refer to $\widetilde{\phi_\textbf{P}}(X)\subset(\p2)^{n}$ as the \emph{images} of $(P_1,\ldots,P_n,X)$.
\end{definition}

Note that every configuration of cameras and points is a reconstruction of its images, so every configuration (of cameras and points) is a reconstruction and vice versa.

Given a set of image points $S\subset(\p2)^{n}$ as well as a reconstruction $(P_1,\ldots,P_n,X)$ of $S$, note that any scaling, rotation, translation, or more generally, any real projective transformation of $(P_1,\ldots,P_n,X)$ does not change the images, giving rise to a large family of reconstructions of $S$. However, we are not interested in differentiating between these reconstructions.

\begin{definition}
\label{def:equvalent_configurations}
Given a set of points $S\subset(\p2)^{n}$, let $(P_1,\ldots,P_n,X)$ and $(Q_1,\ldots,Q_n,Y)$ be two reconstructions of $S$. Let $\overline{X}=\pi_\textbf{P}(X)$ and $\overline{Y}=\pi_{\textbf{Q}(Y)}$ and let $P_i'$ and $Q_i'$ be the matrix representation of $P_i$ and $Q_i$ respectively. The two reconstructions of $S$ are considered \emph{equivalent} if there exists an element $A\in\PGL(4)$, such that
\begin{align*}
&A(\overline{X})=\overline{Y},\\
&P_i'A^{-1}=Q_i',\quad \forall i.
\end{align*}
\end{definition}

From now on, whenever we talk about a configuration of cameras and points, it is to be understood as unique up to the action above, and two configurations will be considered different only if they are nonequivalent. As such, we consider a reconstruction to be unique if it is unique up to the action above.

\begin{definition}
\label{def:conjugate_configuration/point}
Two configurations of cameras and points $(P_1,\ldots,P_n,X)$, and $(Q_1,\ldots,Q_n,Y)$, are called \emph{conjugate configurations} if they are nonequivalent reconstructions of the same set. That is, if they satisfy $\widetilde{\phi_\textbf{P}}(X)=\widetilde{\phi_\textbf{Q}}(Y)$, but not the conditions in \cref{def:equvalent_configurations}. Pairs of points $(x,y)\in X\times Y$ are called \emph{conjugate points} if $\widetilde{\phi_\textbf{P}}(x)=\widetilde{\phi_\textbf{Q}}(y)$.
\end{definition}

\begin{definition}
\label{def:critical_configuration}
A configuration of cameras and points $(P_1,\ldots,P_n,X)$ is said to be a \emph{critical configuration} if it has at least one conjugate configuration. A critical configuration $(P_1,\ldots,P_n,X)$ is said to be \emph{maximal} if there exists no critical configuration $(P_1,\ldots,P_n,X')$ such that $X\subsetneq X'$.
\end{definition}

Hence, a configuration is critical if and only if the images it produces do not have a unique reconstruction.

\begin{remark}
Various definitions of critical configurations exist. For instance, \cite{Krames1941} considers the cone with two cameras on the same generator to be critical, while it fails to be critical by our definition. We use a definition similar to the one in \cite{HK}, except we are working in $\Bl_{\textbf{P}}(\p3)$. If one considers the blow-down, our definition matches the results in \cite{HK}.
\end{remark}

\begin{definition}
\label{def:set_of_critical_points}
Let $\textbf{P}$ and $\textbf{Q}$ be two $n$-tuples of cameras, let $\Bl_{\textbf{P}}(\p3)$ and $\Bl_{\textbf{Q}}(\p3)$ denote the blow-up of $\p3$ in the camera centers of $\textbf{P}$ and $\textbf{Q}$ respectively. Define
\begin{align*}
I=\Set{(x,y)\in\Bl_{\textbf{P}}(\p3)\times\Bl_{\textbf{Q}}(\p3)\mid \widetilde{\phi_\textbf{P}}(x)=\widetilde{\phi_\textbf{Q}}(y)}.
\end{align*}
The projection of $I$ to each coordinate gives us two varieties, $X$ and $Y$, which we denote $X$ as the \emph{set of critical points of $P$ with respect to $Q$}, and similarly for $Y$.
\end{definition}

This definition is motivated by the following fact:

\begin{proposition}
\label{prop:critical_points_give_critical_configuration}
Let $\textbf{P}$ and $\textbf{Q}$ be two $n$-tuples of cameras such that there is no $A\in\PGL(4)$ satisfying $P'_iA^{-1}=Q'_i$ for all $i$. Let $X$ be the critical points of $P$ with respect to $Q$ and conversely for $Y$. Then $(\textbf{P},X)$ is a critical configuration, with $(\textbf{Q},Y)$ as its conjugate. Furthermore, $(\textbf{P},X)$ is maximal with respect to $\textbf{Q}$ in the sense that if there exists a critical configuration $(\textbf{P},X')$ with $X\subsetneq X'$, then its conjugate consists of cameras different from $\textbf{Q}$.
\end{proposition}
\begin{proof}
(2.12 in \cite{twoViews}): It follows from \cref{def:set_of_critical_points} that for each point $x\in X$, we have a conjugate point $y\in Y$. Hence the two configurations have the same images. Inequivalence follows from the fact that the cameras lie in different orbits under $\PGL(4)$, so the second point in \cref{def:equvalent_configurations} can not be satisfied. Hence they are both critical configurations, conjugate to one another. 

The (partial) maximality follows from the fact that if we add a point $x_0$ to $X$ that does not lie in the set of critical points, there is (by \cref{def:set_of_critical_points}) no point $y_0\in\Bl_{\textbf{Q}}(\p3)$ such that $\widetilde{\phi_\textbf{P}}(x_0)=\widetilde{\phi_\textbf{Q}}(y_0)$.
\end{proof}

The goal of this paper is to classify all maximal critical configurations for three cameras. The reason we focus primarily on the maximal ones is that every critical configuration is contained in a maximal one and (when working with more than one camera) the converse is true as well, any subconfiguration of a critical configuration is itself critical.

We conclude this section with a final, useful property of critical configurations, namely that the only property of the cameras we need to consider when exploring critical configurations is the position of their camera centers (i.e. change of coordinates in the images does not affect criticality). 

\begin{proposition}[{\cite[Proposition 3.7]{HK}}]
\label{prop:only_camera_centers_matter}
Let $(P_1,\ldots,P_n)$ be $n$ cameras with centers $p_1,\ldots,p_n$, and let $(P_1,\ldots,P_n,X)$ be a critical configuration. \newline If $(\mathcal{P}_1,\ldots,\mathcal{P}_n)$ is a set of cameras sharing the same camera centers, the configuration $(\mathcal{P}_1,\ldots,\mathcal{P}_n,X)$ is critical as well.
\end{proposition}
\begin{proof}
Since $P_i$ and $\mathcal{P}_i$ share the same camera center and the camera center determines the map uniquely up to a choice of coordinates, there exists some $H_i\in\PGL(3)$ such that $\mathcal{P}_i=H_iP_i$. Let $(Q_1,\ldots,Q_n,Y)$ be a conjugate to $(P_1,\ldots,P_n,X)$. Then $(H_1Q_1,\ldots,H_nQ_n,Y)$ is a conjugate to $(H_1P_1,\ldots,H_nP_n,X)=(\mathcal{P}_1,\ldots,\mathcal{P}_n,X)$, so this configuration is critical as well.
\end{proof}

\section{Approach}
\label{sec:approach}
Consider a critical configuration $(P_1,\ldots,P_n,X)$. Since it is critical, there exists a conjugate configuration $(Q_1,\ldots,Q_n,Y)$ giving the same images in $(\p2)^{n}$. The two sets of cameras define two joint-camera maps $\widetilde{\phi_\textbf{P}}$ and $\widetilde{\phi_\textbf{Q}}$.
\begin{center}
\begin{tikzcd}[ampersand replacement=\&, column sep=small]
\Bl_{\textbf{P}}(\p3) \arrow[rd, "\widetilde{\phi_\textbf{P}}"] \&   \& \Bl_{\textbf{Q}}(\p3) \arrow[ld, "\widetilde{\phi_\textbf{Q}}"'] \\
                          \& (\p2)^{n} \&                           
\end{tikzcd}
\end{center}

The image of $X$ lies in the multi-view variety $\im(\widetilde{\phi_{\textbf{P}}})\subseteq(\p2)^{n}$ and the image of $Y$ lies in $\im(\widetilde{\phi_{\textbf{Q}}})$. As such, the two sets of points $X$ and $Y$ lie in such a way that they both map (with their respective maps) into the intersection of the two multi-view varieties $\im(\widetilde{\phi_{\textbf{P}}})\cap \im(\widetilde{\phi_{\textbf{Q}}})$.

\begin{remark}
This is one of the points at which the choice of coordinates in \cref{rem:basis_chosen} is needed. Without it, there is no way to identify the $(\p2)^{n}$ containing $\im(\widetilde{\phi_{\textbf{P}}})$ with that containing $\im(\widetilde{\phi_{\textbf{Q}}})$.
\end{remark}

\begin{lemma}
\label{lem:joint-camera_map_is_isomorphism}
Let $(P_1,\ldots,P_n)$ be a tuple of $n\geq2$ cameras. If the camera centers do not all lie on a line, the joint camera map $\widetilde{\phi_\textbf{P}}$ is an embedding. If the camera centers all lie on a line, the joint camera map contracts this line and is an embedding everywhere else.
\end{lemma}

\begin{proof}
Let $\textbf{x}=(x_1,\ldots,x_n)\in(\p2)^{n}$ be a point in the image of $\widetilde{\phi_\textbf{P}}$. The preimage of $\textbf{x}$ is the intersection $\bigcap\limits_{i=1}^{n}l_{P_i}$ where $l_{P_i}=P_i^{-1}(x_i)$, that is, the strict transform of a line passing through the camera center $p_i$. As such, they intersect in a single point unless they are all equal to the strict transform of a line passing through all camera centers. Hence the joint camera map is bijective outside of a potential line spanned by the camera centers.

To show that this bijective map is indeed an embedding, we need to show that no tangent collapses. For contradiction, assume there exists a point $x\in\Bl_{\textbf{P}}(\p3)$ and a tangent line $L$ (different from the strict transform of the line containing all camera centers, if such a line exists) which collapses under $\widetilde{\phi_\textbf{P}}$. First, if $L$ is not contained in an exceptional divisor, then $L$ is the strict transform of a line $L'$ in $\p3$. There is at least one camera center, say $p_1$, that does not lie on $L'$. Let $\pi_i$ be the projection onto the $i$-th factor, the map 
\begin{align*}
f\from\Bl_{\textbf{P}}(\p3)\xrightarrow{\widetilde{\phi_\textbf{P}}}(\p2)^{n}\xrightarrow{\pi_1}\p2
\end{align*}
restricted to $L$ is a smooth embedding, of $L$ into $\p2$. Since it factors through $(\p2)^{n}$, the map $\widetilde{\phi_\textbf{P}}$ is also a smooth embedding of this line, which violates our assumption that $L$ is collapsed by $\widetilde{\phi_\textbf{P}}$.

On the other hand, if $L$ is contained in an exceptional divisor, say the one we get when blowing up $p_1$, the map $f$ above is again a smooth embedding when restricted to $L$, by the same argument as above, $\widetilde{\phi_\textbf{P}}$ does not collapse $L$ in this case either. It follows that the map $\widetilde{\phi_\textbf{P}}$ does not collapse any tangents, and hence is an embedding.
\end{proof}

This shows that the two multi-view varieties $\im(\widetilde{\phi_{\textbf{P}}})$ and $\im(\widetilde{\phi_{\textbf{Q}}})$ are irreducible 3-folds, so their intersection is either a surface, a curve, or a set of points. Taking the preimage under $\widetilde{\phi_{\textbf{P}}}$, $\im(\widetilde{\phi_{\textbf{P}}})\cap\im(\widetilde{\phi_{\textbf{Q}}})$ pulls back to a surface, curve, or point set in $\Bl_{\textbf{P}}(\p3)$, which is exactly the set of critical points of $P$. As such, \emph{classifying all maximal critical configurations can be done by classifying all possible intersections between two multi-view varieties.} The configuration of points in $\Bl_{\textbf{P}}(\p3)$ can be found by taking the pullback of $\im(\widetilde{\phi_{\textbf{P}}})\cap\im(\widetilde{\phi_{\textbf{Q}}})$ under $\widetilde{\phi_{\textbf{P}}}$.

\begin{proposition}
\label{cor:critical_configurations_lie_on_quadrics_and_cubics}
Let $(P_1,\ldots,P_n,X)$ be a maximal critical configuration. Then $X$ is the intersection of surfaces $\widetilde{S^{ij}}$ (one for each pair of cameras) and surfaces $\widetilde{S^{ijk}}$ (one or more for each triple of cameras). The surfaces satisfy the following properties:
\begin{itemize}
\item The blow-down $S^{ij}$ of $\widetilde{S^{ij}}$ is a quadric surface containing the camera centers $p_i$ and $p_j$. The blow-down $S^{ijk}$ of $\widetilde{S^{ijk}}$ is a cubic surface containing the camera centers $p_i$, $p_j$, and $p_k$.
\item $\widetilde{S^{ij}}$ is the surface one gets when taking the strict transform of $S^{ij}$ when blowing up $p_i,p_j$ and the total transform under the remaining blow-ups. $\widetilde{S^{ijk}}$ is the surface one gets when taking the strict transform of $S^{ijk}$ when blowing up $p_i,p_j,p_k$ and the total transform under the remaining blow-ups. 
\end{itemize}
\end{proposition}
\begin{proof}
By \cref{thr:ideal_of_multiview_variety}, the ideal of the multi-view variety is generated only by bilinear and trilinear polynomials. Under the rational map $\phi_{\textbf{P}}$, a polynomial in $(\p2)^{n}$ of multidegree $(d_1,\ldots,d_n)$ pulls back to a polynomial of degree $d_1+\cdots+d_n$ with multiplicity $d_i$ in the camera center $p_i$. This gives us the surfaces $S^{ij}$ and $S^{ijk}$ in $\p3$, whose intersection is the blow-down of $X$. 

The variety $X$ itself, however, is the pullback of the same polynomials under the morphism $\widetilde{\phi_{\textbf{P}}}$. Under this morphism, a polynomial of multidegree $(d_1,\ldots,d_n)$ pulls back to a divisor in the class $(d_1+\cdots+d_n)H-d_1E_1-\cdots -d_nE_n$, where $H$ is the class of a hyperplane and $E_i$ is the class of the exceptional divisor we get when blowing up $p_i$. In particular, a bilinear form with $d_i,d_j=1$ will pull back to a surface in the class $2H-E_i-E_j$, that is, the strict transform of a quadric surface under the blow-up of $p_i,p_j$ and the total transform under the remaining blow-ups.
\end{proof}

\begin{remark}
Note that if two sets of cameras satisfy point 2 in \cref{def:equvalent_configurations} their multi-view varieties are equal. Hence the multi-view varieties being different (which is the case we study throughout the paper) is enough to ensure that the two configurations are inequivalent.
\end{remark}

\section{Key results from the two-view case}
\label{sec:two-view_results}
While the main focus of this paper is the critical configurations for three views, there are some key results from the two-view case which can help simplify many of the three-view arguments. This section gives a summary of the main results on critical configurations for two views. A full analysis of the two-view case can be found in \cite{twoViews} or \cite{HK}.

\subsection{The multi-view variety}
\begin{proposition}[The fundamental form]{\cite[Sections 9.2 and 17.1]{hartleyzisserman}}
\label{lem:fundamental_form}
For two cameras $P_1,P_2$ with distinct centers, the multi-view variety $\im(\widetilde{\phi_\textbf{P}})\subset\pxp$ is the vanishing locus of a single, bilinear, rank 2 form $F_P$, called the fundamental form (or fundamental matrix).
\begin{align*}
F_P(\textbf{x},\textbf{y})=\det\begin{bmatrix}
P_1'&\textbf{x}&\textbf{0}\\
P_2'&\textbf{0}&\textbf{y}
\end{bmatrix},
\end{align*}
where $\textbf{x}$ and $\textbf{y}$ are the variables in the first and second image respectively.
\end{proposition}

In the literature, this bilinear form is usually represented by a $3\times3$ matrix called the \emph{fundamental matrix}.

\begin{proof}
By \cref{lem:joint-camera_map_is_isomorphism}, the multi-view variety for two cameras is an irreducible 3-fold in $\pxp$. It follows that the multi-view ideal in the multi-graded ring of $\pxp$ can be generated by a single polynomial. Let $(\textbf{x},\textbf{y})$ be a generic point in the multi-view variety, then there exists a point $\textbf{X}\in\p3$ such that $P_1(\textbf{X})=\lambda_1\textbf{x}$ and $P_2(\textbf{X})=\lambda_2\textbf{y}$, then
\begin{align*}
\underbrace{\begin{bmatrix}
P_1'&\textbf{x}&\textbf{0}\\
P_2'&\textbf{0}&\textbf{y}
\end{bmatrix}}_{A}\begin{bmatrix}
\textbf{X}\\
-\lambda_1\\
-\lambda_2
\end{bmatrix}=0.
\end{align*}
Since $A$ has a non-zero kernel, the determinant $F_P$ has to vanish on $(\textbf{x},\textbf{y})$. Now we need only show that the determinant is irreducible to prove that $F_P(\textbf{x},\textbf{y})$ generates the multi-view ideal. To do this, we show that $F_P$ is of rank 2, this is sufficient since a reducible polynomial is always of rank 1.

Let us find the left kernel of $F_P$, that is, all values of $\textbf{x}$ such that $F_P(\textbf{x},-)=0$. The fifth column of $A$ has zero as its bottom three entries, so the only way to get this as a linear combination of the bottom three entries of the other 5 columns is as $P_2'\cdot p_2+0\cdot \textbf{y}$ (there are no other options since $\textbf{y}$ is a set of variables, and $P_2'$ is of full rank. It follows that the left kernel of $F_P$ consists only of the point $P_1'\cdot p_2$, so $F_P$ is indeed of rank 2.
\end{proof}

\begin{definition}
The special points mentioned at the end of the proof of \cref{lem:fundamental_form}:
\begin{align}
\label{eq:definition_of_epipole}
e_{P_i}^{j}=P_i(p_j),
\end{align}
are called \emph{epipoles}. For the fundamental form of $P_i,P_j$, they satisfy
\begin{align}
\label{eq:epipole_is_kernel}
F_P(e_{P_i}^{j},-)=F_P(-,e_{P_j}^{i})=0.
\end{align}
\end{definition}

By \cref{lem:fundamental_form}, the multi-view ideal for two views is generated by a rank 2 bilinear form. The converse follows from Theorem 9.13. in \cite{hartleyzisserman}.

\begin{theorem}
\label{thr:fundforms_and_camera_pairs_are_1:1}
There is a $1:1$ correspondence between real bilinear forms of rank two, and multi-view ideals for two views.
\end{theorem}

\subsection{Critical quadrics}
Since the multi-view ideal is generated by a bilinear form and since such a form pulls back to the strict transform of a quadric surface passing through the two camera centers in $\p3$, critical configurations for two views need to lie on the strict transform of quadrics (recall \cref{cor:critical_configurations_lie_on_quadrics_and_cubics}). Furthermore, since the bilinear forms vanish in the two epipoles, the quadrics in question need to contain a certain pair of lines (the pullback of each epipole):

\begin{definition}
\label{def:permissible_lines}
Let $p_1,p_2$ be two camera centers, and let $S_P$ be a quadric surface containing both camera centers. A pair of real (not necessarily distinct) lines $\gp{1}{2},\gp{2}{1}$ are called \emph{permissible} if they satisfy the following properties.
\begin{enumerate}
\item The line $\gp{i}{j}$ lies on $S_P$ and passes through $p_i$.
\item Any point lying on both $\gp{i}{j}$ and $\gp{j}{i}$ is a singular point on $S_P$.
\item Any point in the singular locus of $S_P$ which lies on one of the lines also lies on the other.
\item If $S_P$ is the union of two planes, $\gp{i}{j}$ and $\gp{j}{i}$ lie in the same plane.
\end{enumerate}
\end{definition}

This definition is similar to the definition of permissible lines in \cite[Definition 5.11]{HK}, but differs in that it allows for permissible lines in the case where the quadric $S_P$ is a double plane. On the double plane, every point is singular, so two lines are permissible only if they are both equal to the line spanned by the camera centers. \cref{fig:critical_quadrics} shows all quadrics containing at least one pair of permissible lines.

\begin{figure}[]
\begin{center}
\includegraphics[width = \textwidth]{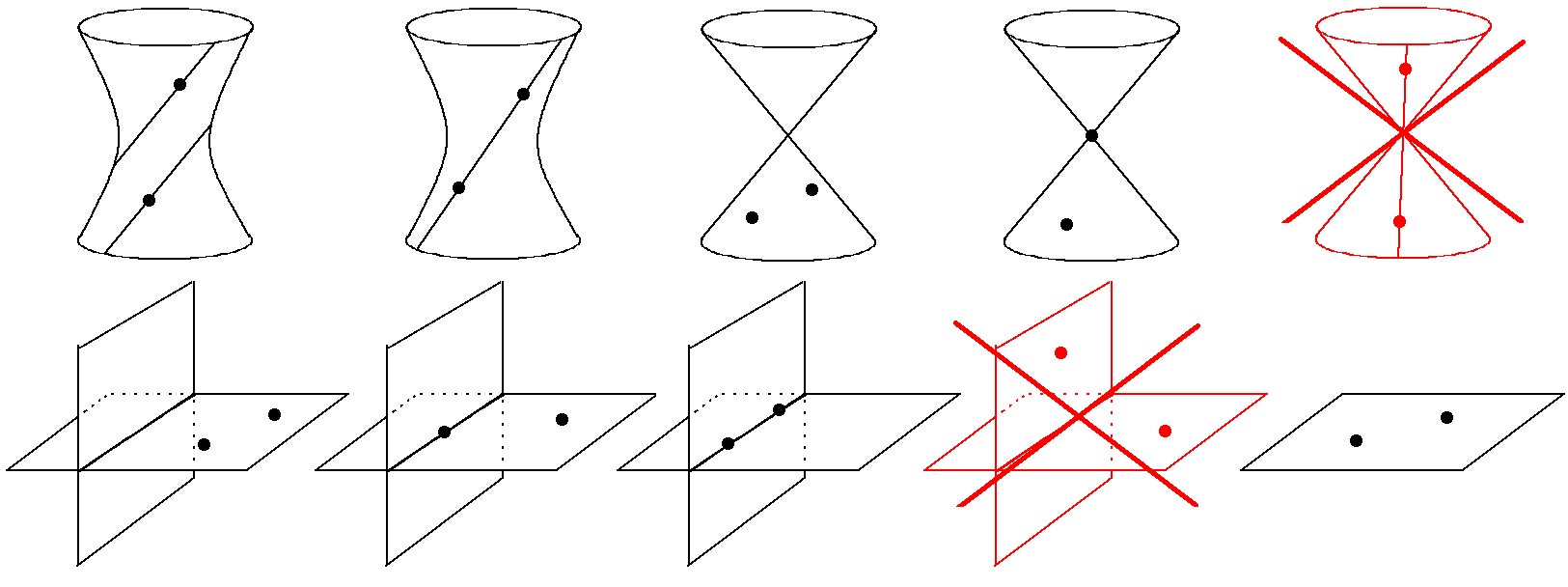}
\end{center}
\caption{Illustration showing all critical quadrics with two marked camera centers, i.e. those containing a pair of permissible lines. The ones marked in red are not critical and do not contain such a pair of lines.}\label{fig:critical_quadrics}
\end{figure}

\begin{theorem}[{\cite[Theorem 4.11]{twoViews}}]
\label{thr:critical_two_views}
Let $(P_1,P_2,X)$ be a configuration of cameras and points. The configuration is critical if and only if there exists a real quadric $S_P\in\p3$ containing both camera centers $p_i$ and a pair of permissible lines, whose strict transform in $\Bl_{\textbf{P}}(\p3)$ contains all the points $X$. 
\end{theorem}
\cite[Theorem 4.11]{twoViews} is formulated slightly differently, requiring the quadrics to be among those in \cref{tab:configurations_and_their_conjugates} rather than to be ones containing two permissible lines. It can easily be checked that these two conditions are equivalent.

\begin{remark}
\label{rem:permissible_are_pullback_of_epipoles}
The theorem requires the quadric to contain a pair of permissible lines for the configuration to be critical. These two lines are the pre-images of the two epipoles $e_{Q_1}^{2},e_{Q_2}^{1}$ of $F_Q$ (where $F_Q$ denotes the fundamental form of the two cameras in the conjugate configuration).
\end{remark}

\begin{proposition}[{\cite[Proposition 5.3]{twoViews}}]
\label{lem:permissible_lines_and_conjugates}
Let $S_P$ be a quadric, and let $P_1,P_2$ be two cameras with camera centers $p_1,p_2$, not both lying in the singular locus of $S_P$. Let $\widetilde{S_P}$ denote the strict transform of $S_P$ in $\Bl_{\textbf{P}}(\p3)$. There is a 1:1 correspondence between pairs of permissible lines on $S_P$, and configurations conjugate to $(P_1,P_2,\widetilde{S_P})$.
\end{proposition}

\begin{table}[]
\begin{tabular}{@{}p{0.54\textwidth}p{0.22\textwidth}p{0.16\textwidth}@{}}
\toprule
\textbf{Quadric $S_P$}      & \textbf{Conjugate quadric} & \textbf{Conjugates}\\
\midrule
Smooth quadric, cameras not on a line       & Same    & 2 \\ \midrule
Smooth quadric, cameras on a line           & Cone, cameras not on a line & 1\\ \midrule
Cone, cameras not on a line & Smooth quadric, cameras on a line           & 1\\ \midrule
Cone, one camera at vertex, other one not & Same  & $\infty$ \\ \midrule
Two planes, cameras in the same plane               & Same   &$\infty$ \\ \midrule
Two planes, one camera on the singular line         & Same   &$\infty$ \\ \midrule
Two planes, cameras on the singular line       & Same   &$\infty$ \\ 
\midrule
A double plane, cameras in the plane     & Same   &$\infty$ \\
\bottomrule
\end{tabular}
\caption{A list of all possible critical configurations and their conjugates, as well as the number of conjugates for each configuration. Taken from {\cite[Table 1]{twoViews}}}
\label{tab:configurations_and_their_conjugates}
\end{table}
Since we will be working with these quadrics extensively throughout the rest of the paper, it makes sense to distinguish the quadrics that give rise to a critical configuration from the ones that do not:

\begin{definition}
Given two cameras $P_1,P_2$, a \emph{critical quadric} is a quadric $S_P$ along with two permissible lines $g_{P_1}^{2}$ and $g_{P_2}^{1}$ such that $(P_1,P_2,\widetilde{S_P})$ is a critical configuration.
\end{definition}

\begin{remark}
\label{rem:quadric_minus_lines}
In \cref{sec:background}, we defined critical configurations to lie in $\Bl_{\textbf{P}}(\p3)$ rather than $\p3$ to ensure that the maximal critical configurations would be closed varieties. Had we defined them to lie in $\p3$, the critical configurations for two views would not be quadrics as above, but rather the quadrics minus the permissible lines. 
\end{remark}

\subsection{Relations between quadrics}
By \cref{thr:critical_two_views}, any critical configuration $(P_1,P_2,X)$ needs to lie on the strict transform of a quadric $S_P$. Similarly, the conjugate configuration $(Q_1,Q_2,Y)$ needs to lie on the strict transform of a quadric (denoted $S_Q$). Each point $x\in S_P$ has a conjugate point $y\in S_Q$ such that $\widetilde{\phi_\textbf{P}}(x)=\widetilde{\phi_\textbf{Q}}(y)$. We give some relations between these conjugate pairs.

\begin{proposition}[{\cite[Table 1]{HK} or {\cite[Table 1]{twoViews}}}]
\label{prop:what_are_the_conjugates}
When the quadric $S_P$ is a hyperboloid with two camera centers on the same line, the conjugate $S_Q$ is a cone with no camera center at the vertex, and vice versa. In all other cases, the quadrics $S_P$ and $S_Q$ are of the same type (see \cref{fig:critical_quadrics} for types). 
\end{proposition}

\cref{thr:critical_two_views} and \cref{lem:permissible_lines_and_conjugates} provide a complete classification of all critical configurations for two views as well as the number of conjugates of each one. These results, together with \cref{prop:what_are_the_conjugates}, are summarized in \cref{tab:configurations_and_their_conjugates}

By \cref{lem:permissible_lines_and_conjugates}, there exists one conjugate configuration for each pair of permissible lines. The choice of lines impacts the relation between points on $\widetilde{S_P}$ and its conjugate $\widetilde{S_Q}$. For the next three results, assume a choice of permissible lines $\gp{1}{2},\gp{2}{1}$ on $S_P$ has been made, and let $\widetilde{\gp{1}{2}}$ and $\widetilde{\gp{2}{1}}$ denote their strict transforms.

\begin{proposition}[{\cite[Lemma 5.4]{twoViews}}]
The map taking a point $x\in\widetilde{S_P}$ to its conjugate is a birational map, defined everywhere except the intersection $\widetilde{\gp{1}{2}}\cap\widetilde{\gp{2}{1}}$.
\end{proposition}

\begin{proposition}[{\cite[Section 5.2]{twoViews}}]
\label{prop:singular_quadric_one_side_means_line_on_the_other}
Points on the strict transform of the line $\gp{i}{j}$ on $\widetilde{S_P}$ map to the exceptional divisor (camera center on the blow-down) $E_{q_j}$ on $\widetilde{S_Q}$. Points lying on the line spanned by $\widetilde{p_1},\widetilde{p_2}$ on $\widetilde{S_P}$ map to the intersection of the two permissible lines $\gq{1}{2},\gq{2}{1}$ on $\widetilde{S_Q}$ (a point which by \cref{def:permissible_lines} is singular).
\end{proposition}

This gives us 2 (or 3, if the line spanned by $p_1$, $p_2$ is contained in $S_P$) curves on $\widetilde{S_P}$ which collapse to points on $\widetilde{S_Q}$. All other lines/curves are mapped to lines/curves.

\begin{definition}
\label{def:type_of_curve}
Let $S_P$ be an irreducible critical quadric and let $C_P$ be a curve on $S_P$. We say that $C_P$ is of type $(a,b,c_1,c_2)$, where $c_i$ is the multiplicity of $C_P$ in the camera center $p_i$ and where $a$ and $b$ are:
\begin{itemize}
\item If $S_P$ is smooth, $a$ is the number of times $C_P$ intersects a generic line in the same family as the permissible lines $g_{P_i}^{j}$, and $b$ is the number of times it intersects the lines in the other family. (meaning $(a,b)$ is the bidegree of $C_P$).
\item If $S_P$ is a cone, $a$ is the number of times $C_P$ intersects each line outside of the vertex, and $b$ the number of times it intersects each line.
\end{itemize}
\end{definition}

\begin{proposition}[{\cite[Proposition 5.6]{twoViews}}]
\label{prop:bidegree_on_P_->_bidegree_on_Q}
Let $S_P$ be a smooth quadric or a cone, and let $C_P\subset \widetilde{S_P}$ be the strict transform of a curve of type $(a,b,c_1,c_2)$ not containing either of the permissible lines. Then the conjugate curve $C_Q\subset \widetilde{S_Q}$ is the strict transform of a curve of type $(a,a+b-c_1-c_2,a-c_2,a-c_1)$.
\end{proposition}

\begin{definition}
Let $S_P$ be two planes with two camera centers not both lying on the intersection and let $C_P$ be a curve on $S_P$. We say that $C_P$ is of type $(a,b,c_0,c_1,c_2)$, where $a$ is the degree of the curve in the plane with the epipolar lines, $b$ the degree of the curve in the other plane, $c_0$ is the multiplicity of $C_P$ in the intersection of the epipolar lines $g_{P_i}^{j}$ and $c_1,c_2$ is the multiplicity in the camera centers $p_1,p_2$ respectively. 
\end{definition}

\begin{proposition}[{\cite[Proposition 5.8]{twoViews}}]
\label{prop:bidegree_on_P_->_bidegree_on_Q_planes}
Let $S_P$ be two planes with two camera centers not both lying on the intersection and let $C_P\subset S_P$ be a curve of type $(a,b,c_0,c_1,c_2)$ such that $C_P$ does not contain either of the epipolar lines or the lines spanned by the camera centers. Then the conjugate curve $C_Q\subset S_Q$ is of type $(2a-c_0-c_1-c_2,b,a-c_1-c_2,a-c_0-c_2,a-c_0-c_1)$.
\end{proposition}

\section{Preliminaries for the three-view case}
\label{sec:preliminary_three_views}
We start with the observation that any critical configuration for three views is critical for each pair of views.
\begin{theorem}
\label{thr:critical_only_if_critical_for_fewer_views}
Let $(P_1,\ldots,P_n,X)$ be a critical configuration, and let $\mathcal{P}\subset\Set{P_1,\ldots,P_n}$ be a nonempty subset. We get the following commutative diagram:
\begin{align*}
\begin{tikzcd}[ampersand replacement=\&]
\Bl_{\textbf{P}}(\p3) \arrow[d, "g_P"] \arrow[rd, "\pi"] \&     \\
\Bl_{\mathcal{P}}(\p3) \arrow[r, "f"]             \& \p3
\end{tikzcd}
\end{align*}
The configuration $(\mathcal{P},g_P(X))$ is either a maximal critical configuration or a subconfiguration of a maximal critical configuration.

\end{theorem}
\begin{proof}
Since $(P_1,\ldots,P_n,X)$ is a critical configuration, there exists a conjugate configuration $(Q_1,\ldots,Q_n,Y)$. Each camera $P_i$ has a conjugate camera $Q_i$. Let $\mathcal{Q}\subset\Set{Q_1,\ldots,Q_n}$ be the set of cameras conjugate to $\mathcal{P}$. Then the cameras $\mathcal{P}$ along with the points 
\begin{align*}
\widetilde{\phi_\mathcal{P}}^{-1}(\im(\widetilde{\phi_\mathcal{P}})\cap\im(\widetilde{\phi_\mathcal{Q}}))
\end{align*}
is a critical configuration with 
\begin{align*}
(\mathcal{Q},\widetilde{\phi_\mathcal{Q}}^{-1}(\im(\widetilde{\phi_\mathcal{P}})\cap\im(\widetilde{\phi_\mathcal{Q}})))
\end{align*} 
as a conjugate. Finally, since $(\mathcal{P},g_P(X))$ and $(\mathcal{Q},g_Q(Y))$ (where $g_Q$ is defined similarly to $g_P$) produce the same image in $(\p2)^{m}$, $g_P(X)$ is either equal to or a subset of $\mathcal{P}^{-1}(\im(\widetilde{\phi_\mathcal{P}})\cap\im(\widetilde{\phi_\mathcal{Q}}))$.
\end{proof}

\begin{definition}
Let $(P_1,\ldots,P_n,X)$ be a configuration of cameras and points, and let $\Set{S_P^{ij}}$ be a set of $n\choose2$ quadrics. Let $\widetilde{S_P^{ij}}$ denote the surface we get in $\Bl_{\textbf{P}}(\p3)$ when taking the strict transform of $S_P^{ij}$ under the blow-up of the camera centers $p_i,p_j$, and the total transform under the remaining blow-ups. The configuration $(P_1,\ldots,P_n,X)$ is said to \emph{lie on the intersection of the quadrics $\Set{S_P^{ij}}$} if each quadric $S_P^{ij}$ contains the camera centers $p_i,p_j$, and $\widetilde{S_P^{ij}}$ contains $X$.
\end{definition}
 
\begin{corollary}
\label{cor:critical_lies_on_intersection_of_quadrics}
For $n\geq2$ views, the critical configurations always lie on the intersection of $n\choose2$ critical quadrics. 
\end{corollary}
\begin{proof}
This follows from \cref{thr:critical_only_if_critical_for_fewer_views} along with the fact that for two views, the critical configurations lie on critical quadrics. It also follows directly from \cref{cor:critical_configurations_lie_on_quadrics_and_cubics}.
\end{proof}

In particular, for the three-view case, the critical configurations lie on the intersection of three critical quadrics. The task of classifying the critical configurations for three views, however, is not as simple as just classifying the possible intersections of the strict transforms of three quadrics. Even if a configuration is critical for each pair of views, it need not be critical for all three views. There are two issues, namely the \emph{compatibility of the quadrics}, and the \emph{ambiguous points}. Before we tackle these issues, let us state a result on the generators of the multi-view ideal for three views.

\begin{lemma}
\label{lem:three_bilinear_and_10_trilinear}
For the three cameras, the multi-view ideal contains a three-dimensional family of bilinear- and a ten-dimensional family of trilinear forms.
\end{lemma}
\begin{proof}
The first part is simply an application of \cref{thr:ideal_of_multiview_variety} to the case of three cameras. To show that the ideal contains a ten-dimensional family of trilinear forms, note that there is a 27-dimensional family of trilinear forms on $\pxpxp$, which pulls back to a 17-dimensional family of cubics passing through the three camera centers (20 cubics, less 3 for the constraint of passing through the camera centers). Hence a 10-dimensional family of trilinear forms pull back to the zero polynomial and thus lie in the multi-view ideal.
\end{proof}

\begin{lemma}
\label{lem:syzygies_iff_colinear}
Let $P_1,P_2,P_3$ be three cameras, and let $F_P^{12},F_P^{13},F_P^{23}$ be their fundamental forms. 
\begin{enumerate}
\item The camera centers are collinear if and only if there are exactly two linear syzygies between the fundamental forms. 
\item The camera centers are not collinear if and only if there are no linear syzygies between the fundamental forms
\end{enumerate}
\end{lemma}
\begin{proof}
\begin{enumerate}
\item Let $P_1,P_2,P_3$ be three cameras whose cameras all lie on a line $l$, and let $\Pi_1$ and $\Pi_2$ be two distinct planes passing through $l$. These planes are mapped to lines in each image, let $f_{P_i}^{j}$ be the polynomial describing the line $P_i(\Pi_j)$, and let $F$ be the matrix
\begin{align*}
F=\begin{bmatrix}
f_{P_1}^{1}&f_{P_2}^{1}&f_{P_3}^{1}\\
f_{P_1}^{2}&f_{P_2}^{2}&f_{P_3}^{2}
\end{bmatrix}.
\end{align*}
Under the joint camera map of the first two cameras $\phi_{P_1,P_2}$, the minor $f_{P_1}^{1}f_{P_2}^{2}-f_{P_1}^{2}f_{P_2}^{1}$ pulls back to the zero polynomial, the only polynomial with this property is the fundamental form $F_P^{12}$. Hence the three minors of $F$ are exactly the fundamental forms $F_P^{ij}$. Since the fundamental forms can be written as the minors of a $2\times3$ matrix, there are two linear syzygies between them.

For the converse, let there be two linear syzygies between the three fundamental forms. This means they can be written as the $2\times2$ minors of a $2\times 3$ matrix
\begin{align*}
F'=\begin{bmatrix}
a&b&c\\
a'&b'&c'
\end{bmatrix}.
\end{align*}
We then have that $F_{P}^{12}=\det\begin{pmatrix}
a&b\\
a'&b'
\end{pmatrix}$ and $F_{P}^{13}=\det\begin{pmatrix}
a&c\\
a'&c'
\end{pmatrix}$. Setting the first column in the first matrix to zero gives us the epipole $e_{P_1}^{2}$ and doing so for the second matrix gives us the epipole $e_{P_1}^{3}$, but since these columns are equal, then so are the epipoles, hence the camera centers must lie on a line. 

\item Assume there exists at least one linear syzygy between the $F_P^{ij}$, and let $\textbf{x}$, $\textbf{y}$, and $\textbf{z}$ be the coordinates in the first, second, and third image respectively. Then there exist three linear polynomials $f_i$ such that
\begin{align*}
f_3(\textbf{z})F_P^{12}(\textbf{x},\textbf{y})+f_2(\textbf{y})F_P^{13}(\textbf{x},\textbf{z})+f_1(\textbf{z})F_P^{23}(\textbf{y},\textbf{z})=0
\end{align*}
This holds for all $\textbf{x}$, $\textbf{y}$ and $\textbf{z}$, in particular, if $\textbf{z}=e_{P_3}^{2}$, we get 
\begin{align*}
f_3(e_{P_3}^{2})F_P^{12}(\textbf{x},\textbf{y})=-f_2(\textbf{y})F_P^{13}(\textbf{x},e_{P_3}^{2}).
\end{align*}
If $F_P^{13}(\textbf{x},e_{P_3}^{2})$ was nonzero, this would imply that $F_P^{12}$ is reducible, and hence of rank 1, meaning it is not a fundamental form. Hence $F_P^{13}(\textbf{x},e_{P_3}^{2})$ must be zero. But this means that $e_{P_3}^{1}=e_{P_3}^{2}$, which means the camera centers are collinear. The converse, that the $F_P^{ij}$ have at least one linear syzygy if the camera centers are collinear follows from item 1. \qedhere
\end{enumerate}
\end{proof}

\begin{proposition}
\label{prop:generators_of_the_MVI}
If the three camera centers $p_1,p_2,p_3$ are not collinear, the multi-view ideal is generated by the fundamental forms along with one trilinear form. If the camera centers lie on a line, the ideal is generated by the fundamental forms along with three trilinear forms.
\end{proposition}
\begin{proof}
By \cref{lem:three_bilinear_and_10_trilinear} the multi-view ideal contains three fundamental forms. Each of these can be used to generate a 3-dimensional family of trilinear forms by multiplying it with the missing variable. By \cref{lem:syzygies_iff_colinear}, if the camera centers are not collinear, there are no linear syzygies between the fundamental forms, so they generate a 9-dimensional family of trilinear forms in the ideal. In this case, we need only add one trilinear form to the set of generators. 

On the other hand, if the camera centers are collinear, we get two linear syzygies, so they generate only a 7-dimensional family of trilinear forms. In this case, we need to add three trilinear forms to get the necessary generators.
\end{proof} 

\subsection{Compatible triples of quadrics}
\label{sec:compatible_triples_of_quadrics}
For three views, the multi-view ideal is generated by three fundamental forms, namely the ones we get from each pair of cameras (as well 1 or 3 trilinear forms, which we consider later). In the two-view case, any fundamental form (bilinear form of rank 2) comes from some pair of cameras (\cref{thr:fundforms_and_camera_pairs_are_1:1}). In the three-view case, however, three arbitrary fundamental forms will generally not come from some triple of cameras. This motivates the following definition:

\begin{definition}
Let $\Set{F_Q^{ij}\mid 1\leq i\leq j\leq n}$ be a set of $n\choose 2$ fundamental forms. The forms are said to be \emph{compatible} if there exist $n$ cameras $Q_1,...,Q_n$ such that $F_Q^{ij}$ is the fundamental form of $Q_i,Q_j$.
\end{definition}

\begin{definition}
Given $n$ cameras $P_1,...,P_n$ we say a set of quadrics $\Set{S_P^{ij}}$ is \emph{compatible} (with respect to the cameras $P_i$), if they are the pullback (with $\phi_{\textbf{P}}$) of a compatible set of fundamental forms.
\end{definition}

For a configuration $(P_1,P_2,P_3,X)$ to be critical, it is not enough that it lies on the intersection of three critical quadrics, the three quadrics also need to be compatible. Only then can we find three conjugate cameras $Q_1,Q_2,Q_3$ which in turn can be used to construct a conjugate configuration. A detailed study of compatible fundamental forms can be found in \cite{compatibility,hartleyzisserman}. Recall (from (\ref{eq:epipole_is_kernel})) that the \emph{epipoles} $e_{Q_i}^{j}$ are the unique points satisfying
\begin{align*}
F_Q^{12}(e_{Q_1}^{2},-)=F_Q^{12}(-,e_{Q_2}^{1})=0.
\end{align*}

\begin{theorem}[{\cite[Proposition 5.8]{compatibility}}]
\label{thr:compatible_forms_collinear}
Let $F_Q^{12}$, $F_Q^{13}$, $F_Q^{23}$ be a triple of rank 2 bilinear forms, and let $e_{Q_i}^{j}$ be their epipoles. Then the triple is compatible and comes from a collinear triple of cameras if and only if
\begin{align}
\label{eq:collinear}
e_{Q_1}^{2}=e_{Q_1}^{3}, \quad e_{Q_2}^{1}=e_{Q_2}^{3}, \quad e_{Q_3}^{1}=e_{Q_3}^{2},
\end{align}
and
\begin{align}
\label{eq:compatible_collinear}
(F_Q^{12})^{T}[e_{Q_1}^{2}]_\times F_Q^{13}=F_Q^{23},
\end{align}
where the product is taken to be matrix multiplication of the matrix representation of the fundamental forms, and where $[a]_{\times}$ denotes the unique $3\times3$ matrix such that $[a]_{\times}b=a\times b$ for all $b$.
\end{theorem} 

\begin{theorem}[{\cite[Section 15.4]{hartleyzisserman}}]
\label{thr:compatible_forms_noncollinear}
Let $F_Q^{12}$, $F_Q^{13}$, $F_Q^{23}$ be a triple of rank 2 bilinear forms, and let $e_{Q_i}^{j}$ be their epipoles. Then the triple is compatible and comes from a non-collinear triple of cameras if and only if
\begin{align}
\label{eq:non_collinear}
e_{Q_1}^{2}\neq e_{Q_1}^{3}, \quad e_{Q_2}^{1}\neq e_{Q_2}^{3}, \quad e_{Q_3}^{1}\neq e_{Q_3}^{2},
\end{align}
and
\begin{align}
\label{eq:compatible_non_collinear}
F_Q^{12}(e_{Q_1}^{3},e_{Q_2}^{3})=F_Q^{13}(e_{Q_1}^{2},e_{Q_3}^{2})=F_Q^{23}(e_{Q_2}^{1},e_{Q_3}^{1})=0.
\end{align}
\end{theorem} 

Since we are working with quadrics in $\p3$, we want to translate these conditions on the fundamental forms, to conditions on the quadrics. The remainder of this subsection is used to give two results stating these conditions. We start with the collinear case, that of \cref{thr:compatible_forms_collinear} (clarification: this is NOT the case where the camera centers $p_i$ are collinear, but the case where the camera centers $q_i$ of the conjugate configuration are collinear).

\begin{proposition}[\textbf{Collinear case}]
\label{prop:compatible_collinear}
Let $P_1$, $P_2$, $P_3$ be a triple of cameras, and let $S_P^{12}$, $S_P^{13}$, $S_P^{23}$ be a triple of critical quadrics. Then the three quadrics come from a triple of fundamental forms that satisfy the conditions in \cref{thr:compatible_forms_collinear}, if and only if there exist three lines $L_{1}$, $L_{2}$, $L_{3}$ satisfying:
\begin{enumerate}
\item $p_i\in L_i$.
\item $L_i\subseteq S_P^{ij}\cap S_P^{ik}$.
\item $L_i$ and $L_j$ form a permissible pair on $S_P^{ij}$.
\end{enumerate}
and any point $x\in S_P^{12}\cap S_P^{13}$ not lying on $L_1$ also lies on $S_P^{23}$.
\end{proposition}
\begin{proof}
($\Rightarrow$):
Let $S_P^{12}$, $S_P^{13}$, $S_P^{23}$ be a compatible (in the collinear sense) triple of critical quadrics. This means they are the pullback of a triple of fundamental forms satisfying the conditions in \cref{thr:compatible_forms_collinear}. By \cref{thr:critical_two_views}, each quadric contains a pair of permissible lines, the pullback of the epipoles. By \cref{eq:collinear} the two epipoles coincide in each image, so the two lines passing through $p_i$ have to coincide as well. This gives us 3 lines satisfying the three conditions above. 

Next, let $x$ be any point in $S_P^{12}\cap S_P^{13}$ not lying on $L_1$, and denote $x_i=P_i(x)$. Since $x$ lies on $S_P^{1i}$, $F_Q^{1i}x_i$ is the unique line in the first image spanned by $x_1$ and the epipole $e_1^{2}=e_1^{3}$. Now, since $F_Q^{12}x_2$ and $F_Q^{13}x_3$ define the same line, we must have
\begin{align*}
0&=x_2^{T}(F_Q^{12})^{T}[e_{Q_1}^{2}]_\times F_Q^{13}x_3,\\
&=x_2^{T}F_Q^{23}x_3,\\
&=x^{T}P_2'^{T}F_Q^{23}P_3'x,\\
&=x^{T}S_P^{23}x.
\end{align*}
It follows that $x$ lies on $S_P^{23}$ as well

($\Leftarrow$):
Let $S_P^{ij}$ be three quadrics satisfying the conditions above. Since $L_i$ and $L_j$ form a permissible pair on $S_P^{ij}$, there exists (by \cref{lem:permissible_lines_and_conjugates}) bilinear forms $F_Q^{ij}$ of rank 2 such that $S_P^{ij}$ is the pullback (under $\phi_{\textbf{P}}$) of $F_Q^{ij}$, and the lines $L_i,L_j$ are the pullback of the epipoles of $F_Q^{ij}$ (see \cref{rem:permissible_are_pullback_of_epipoles}). Since the line $L_i$ is the pullback of both the epipole $e_{Q_i}^{j}$ and the epipole $e_{Q_i}^{k}$, the two epipoles are equal. Hence, the epipoles satisfy \cref{eq:collinear}. 

Let $S_P^{23'}$ denote the quadric defined by the polynomial
\begin{align*}
x^TP_2'^{T}F_Q^{12T}[e_{Q_1}^{2}]_\times F_Q^{13}P_1'x.
\end{align*}
For the next part, we want to show that the $S_P^{23'}=S_P^{23}$. If so, the three quadrics $S_P^{ij}$ are the pullback of the three fundamental forms $F_Q^{12},F_Q^{13}$ and $F_Q^{12T}[e_{Q_1}^{2}]_\times F_Q^{13}$, since this triple of fundamental forms is compatible, so is the triple of quadrics. By assumption, the set $(S_P^{12}\cap S_P^{13})-L_1$ lies on all three quadrics. If the quadrics do not contain a common surface, this set is a (possibly reducible) cubic curve $C$. The curve $C$ lies on both $S_P^{23}$ and $S_P^{23'}$. Moreover, since $F_Q^{12T}[e_{Q_1}^{2}]_\times F_Q^{13}$ has the same left and right kernels as $F_Q^{23}$, $S_P^{23'}$ must also contain the two lines $L_2,L_3$. The union of a cubic curve with two lines uniquely determines a quadric, so we have $S_P^{23}=S_P^{23'}$. A similar argument can be made if the quadrics contain a common surface, and in particular if they are equal.
\end{proof}
\color{black}

Among other things, this shows that as long as the cameras are not in the position shown in \cref{fig:impossible_configurations}, three copies of the same quadric form a compatible triple. For the remainder of the section, we assume that compatible triples of quadrics come from triples of bilinear forms that satisfy the non-collinear conditions (\cref{eq:compatible_non_collinear,eq:non_collinear}), rather than \cref{eq:compatible_collinear,eq:collinear}.

\begin{theorem}[\textbf{General case}]
\label{thr:compatible_quadrics}
Let $P_1,P_2,P_3$ be a triple of cameras, and let $S_P^{12}$, $S_P^{13}$, $S_P^{23}$ be a triple of critical quadrics. Then the three quadrics come from a triple of fundamental forms that satisfy the conditions in \cref{thr:compatible_forms_noncollinear}, if and only if there exist six lines $g_{P_i}^{j}$ such that:
\begin{enumerate}
\item $g_{P_i}^{j}$ and $g_{P_j}^{i}$ form a permissible pair on $S_P^{ij}$ (see \cref{def:permissible_lines}).
\item The intersection of the plane spanned by $g_{P_i}^{j}$ and $g_{P_i}^{k}$ with the plane spanned by $g_{P_j}^{i}$ and $g_{P_j}^{k}$ lies on $S_P^{ij}$.
\end{enumerate}
\end{theorem}

\begin{proof}
($\Rightarrow$):
If the triple of quadrics is compatible, they are the pullback (under $\phi_{\textbf{P}}$) of a compatible triple of rank 2 bilinear forms. The pullback of a rank 2 bilinear form will always be a quadric with two permissible lines, one through each camera center. Hence, the first part is satisfied.

Among the six permissible lines, there are 2 lines on each quadric and 2 through each camera center. By (\ref{eq:non_collinear}), for each camera center, the two lines passing through it are distinct. Hence the two lines through each camera center $p_i$ span a plane, which we denote by $\Pi_i$.

Let $y$ be any point lying in the intersection of $\Pi_{i}$ and $\Pi_j$. Recall \cref{rem:permissible_are_pullback_of_epipoles} which tells us that $P_i'(g_{P_i}^{j})=e_{Q_i}^{j}$. We then have
\begin{align*}
P_i'(y)=\alpha e_{Q_i}^{j}+\beta e_{Q_i}^{k},\\
P_j'(y)=\gamma e_{Q_j}^{i}+\delta e_{Q_j}^{k}.
\end{align*}
for some $[\alpha:\beta],[\gamma:\delta]\in\p1$. Hence
\begin{align*}
S_P^{ij}(y) &=F_Q^{ij}(P_i'(y),P_j'(y))\\
&=F_Q^{ij}(\alpha e_{Q_i}^{j}+\beta e_{Q_i}^{k},\gamma e_{Q_j}^{i}+\delta e_{Q_j}^{k})\\
&=F_Q^{ij}(\beta e_{Q_i}^{k},\delta e_{Q_j}^{k})\\
&=0,
\end{align*}
(final equality follows from (\ref{eq:compatible_non_collinear})) which proves that $y$ lies on $S_P^{ij}$. Hence $S_P^{ij}$ contains the whole intersection of $\Pi_i,\Pi_j$.

($\Leftarrow$): For the converse, assume that the quadrics contain 6 lines satisfying the two conditions. Since $1.$ is satisfied, there exists a $F_Q^{ij}$ such that $S_P^{ij}$ is the pullback of $F_Q^{ij}$, and the lines $g_{P_i}^{j}$ and $g_{P_j}^{i}$ are the pullbacks of the epipoles of $F_Q^{ij}$ (by \cref{rem:permissible_are_pullback_of_epipoles} and \cref{lem:permissible_lines_and_conjugates}). Denote the plane spanned by $\gp{i}{j}$ and $\gp{i}{k}$ by $\Pi_i$.

Let $x$ and $y$ be the two points where $\Pi_i\cap\Pi_j$ intersects $\gp{i}{k}$ and $\gp{j}{k}$ respectively (in the case where $\Pi_i\cap\Pi_j$ is a plane, let $x$ and $y$ be any two points lying on $\gp{i}{k}$ and $\gp{j}{k}$ respectively). If $x$ and $y$ are equal, we get
\begin{align*}
0&=S_P^{ij}(x)\\
&=F_Q^{ij}(P_i'(x),P_j'(y)\\
&=F_Q^{ij}(e_{Q_i}^{k},e_{Q_j}^{k}),
\end{align*}
proving that $F_Q^{ij}$ satisfies (\ref{eq:compatible_non_collinear}). On the other hand, if they are different, they span a line lying in $\Pi_i\cap\Pi_j$, which in turn lies in $S_P^{ij}$. In other words:
\begin{align*}
\alpha x+\beta y \in S_{P}^{ij}
\end{align*}
for all $[\alpha:\beta]\in\p1$. Now
\begin{align*}
P_i'(x)&=e_{Q_i}^{k} &&P_j'(x)=\gamma' e_{Q_j}^{i}+\delta' e_{Q_j}^{k}\\
P_i'(y)&=\gamma e_{Q_i}^{j}+\delta e_{Q_i}^{k} &&P_j'(y)=e_{Q_j}^{k}
\end{align*}
for some $[\gamma:\delta],[\gamma':\delta']\in\p1$, so
\begin{align*}
0&=S_P^{ij}(\alpha x+\beta y)\\
&=F_Q^{ij}(P_i'(\alpha x+\beta y),P_j'(\alpha x+\beta y))\\
&=F_Q^{ij}(\alpha e_{Q_i}^{k}+\beta(\gamma e_{Q_i}^{j}+\delta e_{Q_i}^{k}),\alpha(\gamma' e_{Q_j}^{i}+\delta' e_{Q_j}^{k})+\beta e_{Q_j}^{k})\\
&=F_Q^{ij}((\alpha+\beta\delta)e_{Q_i}^{k},(\alpha\delta'+\beta)e_{Q_j}^{k})\\
&=(\alpha+\beta\delta)(\alpha\delta'+\beta)F_Q^{ij}(e_{Q_i}^{k},e_{Q_j}^{k}).
\end{align*}
This holds for all $[\alpha:\beta]\in\p1$. In particular, we can take $[\alpha:\beta]$ to be such that $(\alpha+\beta\delta)(\alpha\delta'+\beta)\neq0$, proving that $F_Q^{ij}$ satisfies (\ref{eq:compatible_non_collinear}). Hence, $S_P^{12}$, $S_P^{13}$, $S_P^{23}$ forms a compatible triple of quadrics in this case also.
\end{proof}

\subsection{Ambiguous points}
Having dealt with the first issue, the fact that not all triples of quadrics provide three conjugate cameras $Q_1,Q_2,Q_3$, we move on to the second issue, namely that even if a configuration $(P_1,P_2,P_3,X)$ lies on the intersection of three compatible, critical quadrics, it might not be a critical configuration.

By \cref{cor:critical_configurations_lie_on_quadrics_and_cubics} and \cref{prop:generators_of_the_MVI}, the set of critical points is the intersection of the strict transform of three quadrics and one or three cubic surfaces. Hence, there might be some points in the intersection $\widetilde{S_P^{12}}\cap \widetilde{S_P^{13}}\cap \widetilde{S_P^{23}}$ which are not critical.

\begin{proposition}
Given two triples of cameras $(P_1,P_2,P_3)$ and $(Q_1,Q_2,Q_3)$, and a point $x\in \widetilde{S_P^{12}}\cap \widetilde{S_P^{13}}\cap \widetilde{S_P^{23}}$. The three lines
\begin{align*}
l'_{Q_i}=\overline{Q_i'^{-1}(P_i(x))}\subset\p3
\end{align*}
either have a point in common or all lie in the same plane.
\end{proposition}
\begin{proof}
The surface $\widetilde{S_P^{ij}}$ is exactly the collection of points $x$ for which the two lines $l'_{Q_i}$ and $l'_{Q_j}$ intersect one another, so if $x\in \widetilde{S_P^{12}}\cap \widetilde{S_P^{13}}\cap \widetilde{S_P^{23}}$ each pair of lines $l'_{Q_i}$ will intersect. This leaves two options, the $l'_{Q_i}$ are either three lines lying in the same plane, or they have at least one point in common.
\end{proof}

\begin{definition}
Given two triples of cameras, the intersection $\widetilde{S_P^{12}}\cap \widetilde{S_P^{13}}\cap \widetilde{S_P^{23}}$ is the union of two subvarieties: the critical points, for which $l'_{Q_i}$ pass through a common point, and the \emph{ambiguous points}, for which the lines $l_{Q_i}$ are coplanar.
\end{definition}

\begin{remark}
The two subvarieties need not be disjoint, nor do either of them need to be a proper subvariety. For instance, when the camera centers $p_i$ are collinear, the ambiguous points form all of $\widetilde{S_P^{12}}\cap \widetilde{S_P^{13}}\cap \widetilde{S_P^{23}}$ and contain the critical points as a subvariety.
\end{remark}
\begin{remark}
The two subvarieties of $\widetilde{S_P^{12}}\cap \widetilde{S_P^{13}}\cap \widetilde{S_P^{23}}$ come from two subvarieties of $V(F_Q^{12},F_Q^{13},F_Q^{23})$. The critical points come from the multi-view variety, whereas the ambiguous points come from a subvariety that we also refer to as \emph{ambiguous}, making no distinction between $\Bl_{\textbf{P}}(\p3)$ and $\pxpxp$.
\end{remark}

The issue of these ambiguous points is well-known, and descriptions can be found in \cite[Section 15.3]{hartleyzisserman}, and, for a more reconstruction-oriented angle, in \cite[Section 6]{HK}. In this section, we recall a key result by Hartley and Kahl \cite[Theorem 6.17]{HK}, although approached from a more algebraic setting; we also provide a few additional results on the ambiguous points in the case where the cameras are collinear. 

\subsubsection{The non-collinear case}
Let $(Q_1,Q_2,Q_3)$ be three cameras whose camera centers are not collinear. Consider the ideal generated by the three fundamental forms:
\begin{align*}
I=(F_Q^{12},F_Q^{13},F_Q^{23}).
\end{align*} 
This is not the multi-view ideal, since it lacks the final trilinear form (\cref{prop:generators_of_the_MVI}), so $V(I)$ contains at least some points not lying in the multi-view variety. 

\begin{lemma}
\label{lem:V(FFF)_has_two_components}
When the three cameras are non-collinear, the multi-view variety is isomorphic to $\Bl_{\textbf{P}}(\p3)$, while the ambiguous subvariety is isomorphic to $\p1\times\p1\times\p1$. The latter is the product of the three lines we get when projecting the plane spanned by the camera centers $q_1,q_2,q_3$ onto each image.
\end{lemma}
\begin{proof}
The first part is a restatement of \cref{lem:joint-camera_map_is_isomorphism}. As for the ambiguous subvariety, it corresponds to the case where the lines $l'_{Q_i}$ are coplanar. Since $l'_{Q_i}$ passes through $q_i$, the only plane containing these lines is the plane spanned by the camera centers $q_i$. The family of triples of lines lying in that plane describes exactly the $\p1\times\p1\times\p1$ above.
\end{proof}

Adding the final trilinear form to the ideal removes the ambiguous subvariety and leaves only the multi-view variety.

So how does this affect the critical configurations? Let $(P_1,P_2,P_3)$ be a triple of cameras, different from $(Q_1,Q_2,Q_3)$. The pullback (using $\widetilde{\phi_\textbf{P}}$) of the multi-view ideal $\im(\widetilde{\phi_\textbf{Q}})$ is the set of critical points. The pullback of $V(I)$ is the intersection of the strict transforms of three quadrics. This intersection is the union of the critical points and the ambiguous points.

\begin{lemma}
\label{lem:residual_is_intersection_of_three_planes}
Let $P_1$, $P_2$, $P_3$, and $Q_1$, $Q_2$, $Q_3$ be two triples of cameras, such that the camera centers $q_i$ are not collinear. Then the ambiguous points on the $P$-side are the intersection of the strict transforms of three planes, one through each camera center, where the plane through the $i$-th camera center is:
\begin{align*}
\Pi_i&=\vspan(g_{P_i}^{j},g_{P_i}^{k})\\
&=\vspan(P_i^{-1}(e_{Q_i}^{j}),P_i^{-1}(e_{Q_i}^{k}))\\
&=\vspan(P_i^{-1}(Q_i(q_j)),P_i^{-1}(Q_i(q_k))
\end{align*} 
\end{lemma}
\begin{proof}
By \cref{lem:V(FFF)_has_two_components}, the ambiguous subvariety in $V(I)$ is a $\p1\times\p1\times\p1$, the product of the three lines we get by projecting the plane spanned by $q_1,q_2,q_3$ onto each image. In the $i$-th image, this line is the line spanned by the two epipoles $e_{Q_i}^{j}$ and $e_{Q_i}^{k}$, whose pre-image (under $P_i$) is the strict transform of the plane spanned by the two lines $g_{P_i}^{j}$ and $g_{P_i}^{k}$ (recall \cref{rem:permissible_are_pullback_of_epipoles}). Hence, the ambiguous subvariety $\p1\times\p1\times\p1$ pulls back to the intersection of the strict transforms of three planes in $\widetilde{\p3}$.
\end{proof}

\begin{lemma}
\label{lem:remove_line_point_or_plane}
Let $P_1$, $P_2$, $P_3$, and $Q_1$, $Q_2$, $Q_3$ be two triples of cameras, such that the camera centers $q_i$ are not collinear. Then the ambiguous points on the $P$-side make up one of the following varieties:
\begin{itemize}
\item The strict transform of the plane spanned by the three camera centers $p_i$.
\item The strict transform of a line.
\item A single point.
\end{itemize}
In all cases, the variety they form is contained in $\widetilde{S_P^{12}}\cap \widetilde{S_P^{13}}\cap \widetilde{S_P^{23}}$.
\end{lemma}
\begin{proof}
By \cref{lem:residual_is_intersection_of_three_planes}, the ambiguous points are the intersection of the strict transform of three planes, one through each camera center, so they form either a point, a line, or a plane. In the event of a plane, it must be the one spanned by the camera centers. Furthermore, $\p1\times\p1\times\p1$ is a component of $V(I)$, so its pullback is contained in the pullback of $V(I)$, which is the intersection $\widetilde{S_P^{12}}\cap \widetilde{S_P^{13}}\cap \widetilde{S_P^{23}}$.
\end{proof}

\begin{theorem}[{\cite[{Theorem 6.17}]{HK}}]
\label{thr:compatible_implies_critical}
Let $P_1,P_2,P_3$ be a triple of cameras, let $S_P^{12}$, $S_P^{13}$, $S_P^{23}$ be a compatible (in the non-collinear sense) triple of quadrics, and let $X$ be the intersection of their strict transforms. The compatible triple of quadrics comes with six permissible lines, spanning three planes. Let $x$ be the intersection of the strict transforms of these three planes. Then the cameras $P_i$ along with the points in $\overline{X-x}$ constitute a critical configuration.
\end{theorem}

\begin{proof}
Since the three quadrics $S_P^{ij}$ are compatible, there exist three non-collinear cameras $Q_1,Q_2,Q_3$ giving rise to these quadrics. Now by \cref{lem:remove_line_point_or_plane}, this means that all of $X$, with the possible exception of a point or the strict transform of a line or plane, lies in the set of critical points. This point/line/plane is exactly the intersection of the three planes spanned by the six lines. Hence the set $X-x$ lies in the set of critical points. By continuity, this is also the case for $\overline{X-x}$, meaning that the cameras $P_i$ along with the points in $\overline{X-x}$ constitute a critical configuration.
\end{proof}

\begin{remark}
In the case where $x$ is the strict transform a plane, it might happen that the configuration $(P_1,P_2,P_3,\overline{X-x})$ is not a maximal critical configuration, as the intersection of the critical and ambiguous points need do not lie in the closure of $X-x$. For instance, in the case where the three quadrics are all equal to the union of two planes, $x$ is the strict transform of the plane containing the camera centers. $(P_1,P_2,P_3,\overline{X-x})$ is not a maximal critical configuration in this case, as there is a conic curve in $x$ which also lies in the set of critical points (see \cref{prop:plane+conic_critical}).
\end{remark}

\cref{thr:compatible_implies_critical} tells us all we need to classify the (non-collinear) critical configurations for three views. We need only check in what ways three compatible quadrics can intersect, and then remove the strict transform of a point/line/plane if needed. In particular, if the intersection $S_P^{12}\cap S_P^{13}\cap S_P^{23}$ does not contain a plane, a line, or any isolated points, the whole intersection is part of a critical configuration.

\subsubsection{The collinear case}
We now consider the other case, where $(P_1,P_2,P_3)$ and $(Q_1,Q_2,Q_3)$ are two triples of cameras, such that the camera centers $q_1,q_2,q_3$ lie on a line. In this case,
\begin{align*}
V(I)=V(F_Q^{12},F_Q^{13},F_Q^{23})\subset\pxpxp
\end{align*} 
is still the union of the same two subvarieties: the multi-view variety, and the ambiguous one. In this case, however, the latter is not a $\p1\times\p1\times\p1$ but rather all of $V(I)$, containing the multi-view variety as a divisor. Indeed, whenever the lines $l'_{Q_i}$ intersect in a point, the three lines are coplanar, since the camera centers $q_i$ lie on a line.

Since the set of critical points is contained in the ambiguous points, it is not enough to simply remove a point, line, or plane. It follows that \cref{thr:compatible_implies_critical} does not hold in the case where the cameras on the other side are collinear. Unlike the non-collinear case, the permissible lines on the quadrics $S_P^{ij}$ are not enough to determine the critical points. We do, however, have the following result, which holds for both the collinear and non-collinear cases:
\begin{proposition}
\label{prop:reconstructible_codim_1_in_residual}
The intersection of the multi-view variety with the ambiguous points forms a codimension 1 subvariety of the latter.
\end{proposition}
\begin{proof}
In the family of triples of coplanar lines, requiring the third line to pass through the intersection of the first two is (generally) a linear constraint. Hence the triples of lines all having at least one point in common form a codimension 1 subvariety of this family.
\end{proof}

Hence, the ambiguous points in $S_P^{12}\cap S_P^{13}\cap S_P^{23}$ contain a subvariety that also lies in the set of critical points. This subvariety has codimension at most 1 (exactly 1 if we assume that the ambiguous points are not contained in the set of critical points).

\section{Critical configurations for three views}
\label{sec:critical_configurations_for_three_views}
The goal of this section is to provide proof of \cref{thr:main_theorem}. By \cref{cor:critical_lies_on_intersection_of_quadrics}, and the discussion in \cref{sec:compatible_triples_of_quadrics}, any critical configuration for three views lies on the intersection of three compatible quadrics. Furthermore, any such intersection is critical, with the possible exception of some number of ambiguous points (usually a point, line, or plane). As such, our approach is to consider all possible intersections of three quadrics, and for each case check whether it exists as the intersection of a compatible triple and if so, remove any potential ambiguous points. This will give us all critical configurations for three views. We will start with the cases where the quadrics intersect in a surface (\cref{ssec:quadrics_intersecting_in_surface}), then a curve (\cref{sec:quadrics_intersecting_in_a_curve}), and finally in a finite number of points (\cref{sec:quadrics_intersecting_in_points}).

\begin{remark} 
While the intersection of quadrics may contain non-reduced components, the definition of critical configuration is a set-theoretical one. This means that the multiplicity of components does not affect criticality. As such, we consider all components to be reduced.
\end{remark}
\begin{remark}
While the critical configurations for two views tend to have a finite number of conjugates, this is generally not the case for three views. In fact, the only critical configurations for three views that have a finite number of conjugates are the ones consisting of 7 or fewer points. The rest all have an infinite number of conjugates (although the dimension of the family of conjugates may vary). The proofs will usually revolve around constructing a single conjugate, but the choices made in the construction should make it clear that many conjugates exist when different choices are made.
\end{remark}

\subsection{Quadrics intersecting in a surface}
\label{ssec:quadrics_intersecting_in_surface}
In the two-view case, the points in a maximal critical configuration always form a surface. We begin by checking whether any such configurations remain critical in the three-view case. In particular, we want to examine whether there exist any critical configurations where the points form a surface. There are two ways three quadrics can intersect in a surface, either by all being equal or by all being reducible and sharing a common plane.

\subsubsection{The smooth quadric/cone}
\label{sec:all_quadrics_coincide}
We start with the case where the quadrics $S_P^{ij}$ are all the same smooth quadric or cone (denoted simply $S_P$). In this case, it is impossible to find 6 lines satisfying the conditions in \cref{thr:compatible_quadrics}, so one cannot find a non-collinear conjugate triple of cameras $Q_1,Q_2,Q_3$ such that $\widetilde{S_P}$ is the pullback of all three fundamental forms $F_Q^{ij}$. However, as long as $S_P$ is not one of the following configurations (illustrated in \cref{fig:impossible_configurations})
\begin{itemize}
\item $S_P$ is smooth and there is a camera center such that the two lines passing through it, each pass through one other camera center.
\item $S_P$ is a cone and two camera centers lie on the same line, but neither of the two lies at the vertex
\end{itemize} 
then three copies of $S_P$ satisfy the conditions of \cref{prop:compatible_collinear}, and hence form a compatible triple, coming from a triple of collinear cameras $Q_i$.

\begin{figure}[]
\begin{center}
\includegraphics[width = 0.50\textwidth]{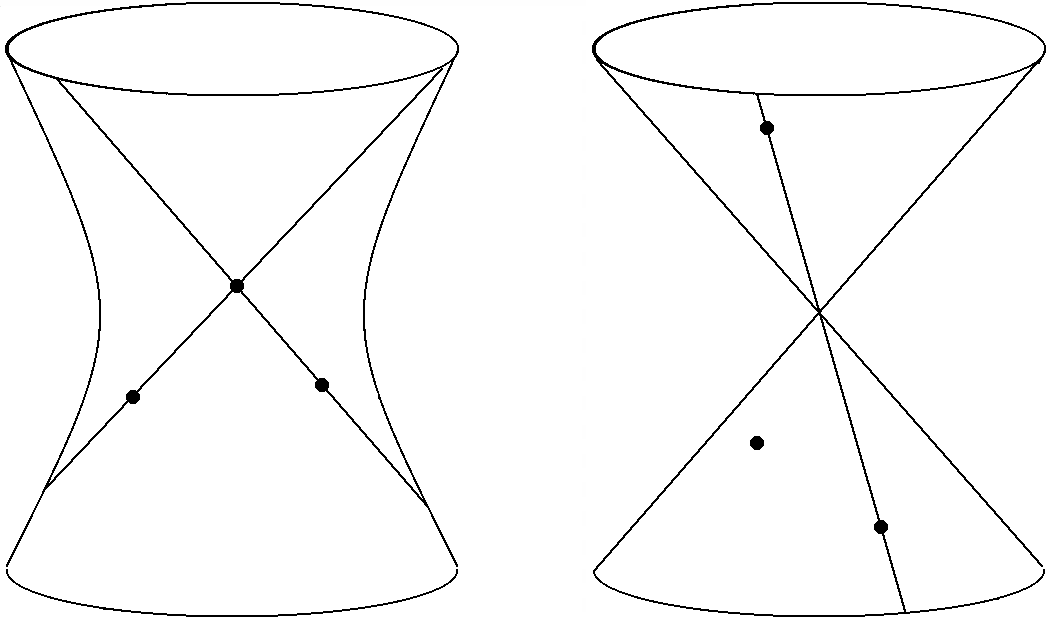}
\end{center}
\caption{If the camera centers are positioned as shown, the three quadrics (all three equal) do not constitute a compatible triple.}\label{fig:impossible_configurations}
\end{figure}

\begin{proposition}
\label{prop:rational_curves}
Let $P_1,P_2,P_3$ be three cameras, and let $S_P^{12},S_P^{13},S_P^{23}$ be a compatible triple of quadrics. If the three quadrics are all equal to the same irreducible quadric $S_P$, the set of critical points $X$ is the strict transform of a rational curve of degree at most four passing through all camera centers. The conjugate configuration is the strict transform of a curve of one degree less, not passing through any camera centers (plus the strict transform of the line passing through the camera centers, if $S_P$ is a cone).
\end{proposition}

\begin{remark} 
Note that the two cases in \cref{fig:impossible_configurations} are exactly the two where $S_P$ does not contain a rational curve passing through all three camera centers
\end{remark}

\begin{proof}
Since the three cameras $Q_i$ on the other side have to be collinear, our critical configuration does not consist of all the points on the strict transform of $S_P$. By \cref{prop:reconstructible_codim_1_in_residual}, the critical points on $\widetilde{S_P}$ form a codimension 1 subvariety, the strict transform of a curve $C_P\subset S_P$.

The three cameras $Q_i$ lie on a line, so by \cref{prop:generators_of_the_MVI}, the ideal of their image in $\pxpxp$ is generated by three bilinear and three trilinear forms $(F_Q^{12},F_Q^{13},F_Q^{23},T_1,T_2,T_3)$. The fundamental forms pull back to quadrics $S_P^{ij}$ through two camera centers, while the trilinear forms pull back to cubic surfaces through all three camera centers. Hence the curve $C_P$ is the intersection of $S_P$ with three cubic surfaces, all containing the three camera centers. It follows that $C_P$ is of type $(a,b,c_1,c_2)$ (recall \cref{def:type_of_curve}) with $a,b\leq3$ and $c_1,c_2\geq1$.

The curve $C_P$ lies on all three quadrics $S_P^{ij}$ and its type $(a,b,c_i,c_j)$ depends on which quadric we consider since it is defined by the choice of permissible lines and the number of times it intersects two of the three camera centers. The permissible lines lie in the same family on all three quadrics, so $a$ and $b$ are independent of the choice of quadric. Furthermore, the total degree of the conjugate curve $C_Q$ depends (by \cref{prop:bidegree_on_P_->_bidegree_on_Q}) on the multiplicity of $C_P$ in the camera centers. Since the curve $C_Q$ is the same regardless of which quadric we consider, it follows that $C_P$ needs to have the same multiplicity in all three camera centers, the same is true for $C_Q$.

If we assume that the three camera centers $p_i$ are not collinear, the three quadrics $S_Q^{ij}$ on the other side cannot be the same quadric. It follows that the conjugate curve $C_Q\subset S_Q^{12}\cap S_Q^{13}\cap S_Q^{23}$ is of type $(a',b',c_1',c_2')$ with $a',b'\leq2$. If, on the other hand, the camera centers $p_i$ are collinear, the resulting curve is still the intersection of quadrics and cubics, and so needs to satisfy $a',b'\leq3$.

Furthermore, if $S_P$ is smooth, the curve $C_Q$ cannot pass through any of the three camera centers $q_i$, since this would imply that all three quadrics $S_Q$ contain the line spanned by the camera centers, which (by \cref{prop:what_are_the_conjugates}) violates our condition that $S_P$ is smooth.

On the other hand, if $S_P$ is a cone, then the three quadrics $S_Q^{ij}$ must (by \cref{prop:what_are_the_conjugates}) contain the line $l_Q$ spanned by the camera centers. However, the whole line $l_Q$ is conjugate to a single point, namely the vertex of $S_P$ (by \cref{prop:singular_quadric_one_side_means_line_on_the_other}). So in the case where $S_P$ is a cone, $C_Q$ is the union of $l_Q$ and some other curve.

This leaves only a finite number of options for what kind of curve $C_P$ can be. Finally, the curves should satisfy the result in \cref{prop:bidegree_on_P_->_bidegree_on_Q}: that if $C_P$ is of type $(a,b,c_1,c_2)$, then $C_Q$ is of type $(a,a+b-c_1-c_2,a-c_2,a-c_1)$. We can now check every possible type for $C_P$ and remove all cases where $a'$ or $b'$ are either negative or greater than 3, the cases where $c_i'$ is negative, and the cases where the multiplicity in the camera centers is greater than what the degree of the curve allows (for instance, a curve of degree 1 can not have multiplicity 2 in a camera center).

This leaves only three options, listed in \cref{tab:curves_on_same_quadric}. Note that if $S_P$ is a cone, then the curve on the other side will also include the line $l_Q$ in addition to the component given in the table.
\end{proof}

\begin{table}[h]
\centering
\begin{tabular}{@{}p{0.18\textwidth}p{0.18\textwidth}@{}}
\toprule
Curve on $S_P$ & Curve on $S_Q$ \\ \midrule
(1,3,1,1)         & (1,2,0,0)          \\ 
(1,2,1,1)         & (1,1,0,0)          \\ 
(1,1,1,1)         & (1,0,0,0)          \\ \bottomrule
\end{tabular}
\caption{Recall that a curve of type $(a,b,c_1,c_2)$ intersects the family containing the permissible lines $a$ times, the lines in the other family $b$ times, and passes $c_i$ times through the camera center $p_i$.}
\label{tab:curves_on_same_quadric}
\end{table}

\begin{remark}
This proves that any critical configuration where the three quadrics on one side coincide has to be one of the types in \cref{tab:curves_on_same_quadric}. In \cref{sec:quadrics_intersecting_in_a_curve} we prove the converse, that any configuration of such type is critical. The general case is the rational quartic on one side, and a twisted cubic on the other, while the two other configurations appear as degenerates of this one. 
\end{remark}

\subsubsection{Two planes}

\begin{proposition}
\label{prop:plane+conic_critical}
When the three quadrics $S_P^{ij}$ are all equal to the same two planes, the set of critical points is the union of a plane and the strict transform of a conic passing through all three camera centers. Conversely, any configuration where the points lie on the union of a plane and the strict transform of a conic curve passing through the camera centers is a critical configuration. Its conjugate is of the same type (plane + conic curve).
\end{proposition}

\begin{proof}
Let the three quadrics $S_P^{ij}$ all be equal to the same two planes $S_P$, \cref{thr:critical_two_views} states that all three cameras need to lie in the same plane, with any number of them lying on the intersection of both. In this case, unlike the previous one, there exist six lines satisfying the conditions in \cref{thr:compatible_quadrics} as long as not all three cameras lie on the intersection of the two planes. An example is given in \cref{fig:two_planes}. Hence, the three cameras $Q_i$ on the other side can be chosen to not be collinear. 

Since the three cameras $Q_i$ on the other side can be assumed to not be collinear, \cref{thr:compatible_implies_critical} proves that the plane not containing all camera centers is part of the set of critical points. Furthermore, \cref{prop:reconstructible_codim_1_in_residual} proves that there is a curve $C_P$ in the plane with the camera centers whose strict transform is also part of this set. By \cref{prop:generators_of_the_MVI}, the set of critical points is the strict transform of the intersection of the two planes with a cubic surface passing through the three camera centers, so the curve $C_P$ is a conic passing through the three camera centers. 

For the converse, we want to determine the conic curve $C$ whose strict transform lies in the set of critical points. We already know that $C_P$ passes through the three camera centers. Furthermore, the intersection of the two permissible lines $g_{P_i}^{j}$ and $g_{P_k}^{j}$ (denote it by $a_j$) must also lie in the set of critical points since it has a conjugate point lying on the exceptional divisor $E_{q_j}$. This gives us three more points on the conic. Usually, there is no conic passing through six given points. In this case, however, the six points $p_1,p_2,p_3,a_1,a_2,a_3$ along with the six lines $g_{P_i}^{j}$ span a hexagon whose opposite sides ($g_{P_i}^{j}$ and $g_{P_j}^{i}$) meet at three points which lie on a line (the intersection of the two planes) so by the Braikenridge–Maclaurin theorem \cite[p.76]{conversePascal} the six points $p_1,p_2,p_3,a_1,a_2,a_3$ lie on a conic.

Finally, for any conic $C_P$ through $p_1,p_2,p_3$ one may choose the six lines $g_{P_i}^{j}$ in such a way that the three points $a_i$ lie on $C_P$ \emph{and} each pair of lines form a permissible pair. This can be done by picking $a_1$ to be any\footnote{not equal to a camera center and not lying in the intersection of the two planes} point on $C_P$ and then taking $a_2$ to be the intersection of $g_{P_1}^{2}$ with $C_P$. Since we can get any conic by making the appropriate choice of permissible lines, this proves that any plane + conic configuration is critical. 
\end{proof}

\begin{figure}[]
\begin{center}
\includegraphics[width = 0.60\textwidth]{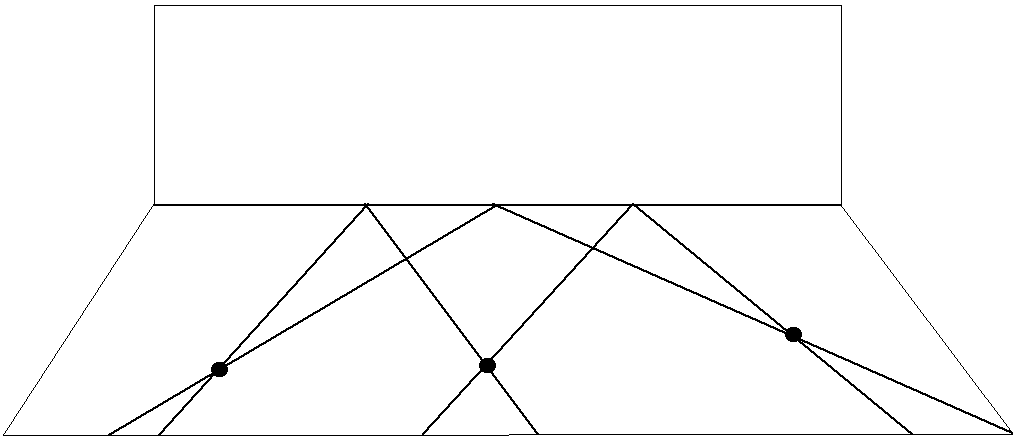}
\end{center}
\caption{One example of six lines satisfying the conditions in \cref{thr:compatible_quadrics}}\label{fig:two_planes}
\end{figure}

\subsubsection{Single plane}
With this, there is only one way left that the three quadrics $S_P^{ij}$ can intersect in a surface, namely if they are all reducible, and contain only one common plane. In this case, the set of critical points is at most a plane plus a line. Hence, this is not a maximal critical configuration, as it is always a subconfiguration of the plane + conic configuration in \cref{prop:plane+conic_critical}.

\subsection{Quadrics intersecting in a curve}
\label{sec:quadrics_intersecting_in_a_curve}
Curves that appear as the intersection of quadrics have at most degree 4. Moreover, any curve that lies on the union of a plane and a conic passing through the camera centers is critical by \cref{prop:plane+conic_critical}, which further reduces the number of curves we need to check. We analyze all curves appearing as the intersection of quadrics sorted by the degree of their highest degree component. In \cref{sec:quartic_curves}, we analyze the curves where the highest degree component is of degree 4, covering the elliptic quartic as well as the singular rational quartics. \cref{sec:cubic_curves} covers the curves containing a twisted cubic, while \cref{sec:curves_of_lower_degree} covers the curves containing no components of degree higher than 2.

\begin{lemma}
\label{lem:intersection_of_strict_transforms_is_strict_transform_of_intersection}
Let $S_P^{ij}$ be a triple of quadrics in $\p3$ intersecting in a curve $C_P$. Let $\widetilde{S_P^{ij}}$ be the surface we get in $\Bl_{\textbf{P}}(\p3)$ when taking the strict transform under the blow-up in $p_i,p_j$ and the total transform under the final blow-up. Then the intersection $\widetilde{S_P^{12}}\cap \widetilde{S_P^{13}}\cap \widetilde{S_P^{23}}$ is the strict transform of $C_P$.
\end{lemma}
\begin{proof}
The statement is trivial when $C_P$ does not pass through a camera center. If $C_P$ does pass through a camera center, any tangent to $C_P$ in this camera center must also be a tangent to each of the quadrics $S_P^{ij}$, it follows that the intersection of the $\widetilde{S_P^{ij}}$ is the strict transform of $C_P$.
\end{proof}

While it is not true in general that the intersection of strict transforms is the strict transform of the intersections, \cref{lem:intersection_of_strict_transforms_is_strict_transform_of_intersection} tells us that this is the case for three quadrics intersecting in a curve. This allows us to work in $\p3$, considering all curves appearing as intersections of three quadrics. The critical configurations will then consist of points lying on the strict transforms of whichever curves we have in $\p3$.

For most curves, the main issue is to find a compatible (in the non-collinear sense) triple of quadrics containing it. The issue of having to remove a point, line, or plane (as per \cref{thr:compatible_implies_critical}) is less relevant since the curves will usually not contain any such component. To find compatible triples of quadrics, the following lemmas are useful:

\begin{lemma}
\label{lem:compatible}
Let $P_1$, $P_2$, $P_3$ be a triple of cameras, and let $S_P^{12}$, $S_P^{13}$, $S_P^{23}$ be a triple of critical quadrics, each passing through two camera centers (indicated by the superscript). Then the three quadrics are compatible if
\begin{enumerate}
\item[A)] The intersection of $S_P^{ij}$ and $S_P^{ik}$ does not contain a line passing through $p_i$.
\item[B)] There exists a point $x$ lying on all three quadrics, as well as three distinct lines $l_{12}$, $l_{13}$, $l_{23}$ passing through $x$, satisfying
\begin{enumerate}
\item[1.] $l_{ij}\subset S_P^{ij}$.
\item[2.] $p_i\in \vspan(l_{ij},l_{ik})$.
\item[3.] $l_{ij}\notin \vspan(l_{ik},l_{jk})$.
\item[4.] $p_i,p_j\notin l_{ij}.$
\end{enumerate}
\end{enumerate}
\end{lemma}

\begin{proof}
Given three quadrics $S_P^{ij}$, assume there exists a point $x$ and three lines $l_{ij}$ satisfying the conditions in the lemma above. For each quadric $S_P^{ij}$, pick two lines $g_{P_i}^{j}$ and $g_{P_j}^{i}$ such that they both intersect $l_{ij}$, and form a permissible pair. The only case where such a choice is not possible is if the quadric $S_P^{ij}$ contains the line spanned by $p_i$ and $p_j$, and the line $l_{ij}$ intersects this line, but this would violate either condition 2, 3 or 4.

This gives us six lines $g_P$, two on each quadric, and two through each camera center. Condition $A)$ ensures that the two lines through each camera center are distinct, so they span a plane. By construction, the plane spanned by $l_{ij}$ and $l_{ik}$ is the same as the one spanned by $\gp{i}{j}$ and $\gp{i}{k}$. This means that the lines $g_P$ satisfy the conditions in \cref{thr:compatible_quadrics}. It follows that the three quadrics are compatible.
\end{proof}

\begin{lemma}
\label{lem:compatible_line}
Let $P_1$, $P_2$, $P_3$ be a triple of cameras, and let $S_P^{12}$, $S_P^{13}$, $S_P^{23}$ be a triple of critical quadrics. If there exists a line $L$ lying on all three quadrics, but not passing through the camera centers such that the intersection of $S_P^{ij}$ and $S_P^{ik}$ does not contain a line passing through $p_i$ and intersecting $L$, then the quadrics form a compatible triple.
\end{lemma}

\begin{proof}
On each quadric $S_P^{ij}$, take $g_{P_i}^{j}$ to be the line through $p_i$ intersecting $L$. By the same arguments as above, this gives us 6 lines satisfying \cref{thr:compatible_quadrics}.
\end{proof}

\begin{remark} 
In the two lemmas above, the lines $l_{ij}$ and $L$ all intersect the permissible lines, meaning they are of type $(1,0,0,0)$ on their respective quadrics.
\end{remark}

\begin{remark}
The two lemmas above have the condition that the intersection of two quadrics does not contain a line passing through their common camera center. This is to ensure that the two lines $g_{P_i}^{j}$ and $g_{P_i}^{k}$ span a plane, which is a necessary condition for the quadrics to be compatible (in the non-collinear sense). This is a weak condition, and in the following sections, whenever we apply either of the two lemmas above, the cases where two lines through a camera center coincide can easily be avoided, if such cases exist at all. To keep the arguments relatively short, the verification of this fact is left to the reader.
\end{remark}

\subsubsection{Quartic curves}
\label{sec:quartic_curves}
We begin with the case where the three quadrics all lie in the same pencil, that is, where the intersection of all three is equal to the intersection of any two. In this case, the intersection is a quartic curve, passing through at least two of the camera centers. If the curve is irreducible, it is either a smooth elliptic quartic curve or a singular rational quartic (cusp or node). The cases where the curve is reducible are covered in \cref{sec:cubic_curves,sec:curves_of_lower_degree}.

\begin{proposition}
\label{prop:elliptic_curve_is_critical}
Any configuration consisting of three cameras along with a set of points lying on the strict transform of an elliptic quartic curve passing through the camera centers is critical. The same is the case if the points lie on the strict transform of a singular quartic curve passing through the camera centers, as long as neither camera center lies on the singular point. In both cases, they have conjugates that are curves of the same type.
\end{proposition}

\begin{proof}
A quartic curve (elliptic or singular) is contained in a pencil of quadrics. This leaves 1 degree of freedom when choosing quadrics. In particular, for each line secant to the curve, there is a unique quadric in the pencil that contains this line. 

Denote the quartic curve by $C$, and let $x$ be a smooth point on $C$ not lying in the plane spanned by $p_1$, $p_2$, $p_3$. Projecting from the point $x$, we get a cubic curve $C'$ in the plane, along with three distinct points $p_1',p_2',p_3'$ on the curve (the projection of the camera centers).

Assume there exist three real points $y_1,y_2,y_3$ on this curve, different from the $p_i'$, such that $p_i'$ lies on the line $\overline{y_jy_k}$ for $i\neq j\neq k$ (see \cref{fig:22111}).

\begin{figure}
\begin{center}
\includegraphics[width = 0.55\textwidth]{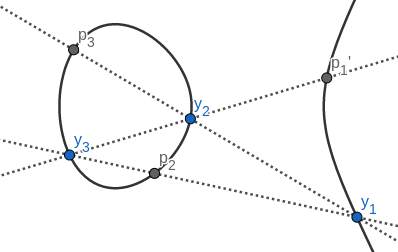}
\end{center}
\caption{Given three points $p_1',p_2',p_3'$ on a smooth cubic plane curve, there exist three lines forming a triangle with vertices $y_1,y_2,y_3$ such that each line contains exactly one point.}\label{fig:22111}
\end{figure}

The preimages of the three points $y_i$ are three lines through $x$, secant to $C$. Denote the preimage of $y_i$ by $l_{jk}$, and let $S_P^{ij}$ be the unique quadric containing $C$ and the secant $l_{ij}$. Then the point $x$ along with the three lines $l_{12},l_{13},l_{23}$ satisfy the conditions in \cref{lem:compatible}, so the three quadrics $S_P^{12},S_P^{13},S_P^{23}$ constitute a compatible triple. Since the curve $C$ is the intersection of three compatible quadrics, the three cameras along with the strict transform of $C$ constitute a critical configuration.

Now we need only prove that three such points $y_i$ exist. Given the three points $p_i'$ on $C'$, define a map 
\begin{align*}
f\from C'\dashrightarrow C'
\end{align*}
as follows: for each smooth point $a$, the lines $\overline{p_1'a}$ and $\overline{p_2'a}$ intersect the cubic in two new points, which we denote by $b_1$ and $b_2$ respectively. The line $\overline{b_1b_2}$ intersects $C$ in a third point, let this point be $f(a)$ (see \cref{fig_operation}). This construction yields a rational map that extends to a morphism by considering $\overline{p_i'a}$ to be the tangent line to $C$ in $p_i'$ if $a=p_i'$ (this only gives a morphism in the smooth case, in the singular it remains undefined in the singular point).

\begin{figure}[h]
\begin{center}
\includegraphics[scale=0.5]{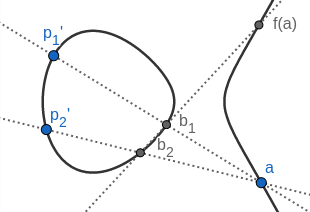}
\end{center}
\caption{Illustration of the map $f$ on a smooth cubic curve}
\label{fig_operation}
\end{figure}

If there exists a point $a$ such that $f(a)=p_3'$, we can take $y_1,y_2,y_3$ to be the points $a,b_1,b_2$, since these satisfy the conditions above. The map $f$ can be expressed as 
\begin{align*}
f(a)=p_1'+p_2'+2a
\end{align*}
where \enquote{+} is the sum under the usual group law on an elliptic curve (and a similarly defined sum on the singular one). The fiber over $p_3'$ consists of all points $a$ satisfying
\begin{equation}
\label{eq_gruppevirkning}
2a=p_3'-p_1'-p_2'
\end{equation}
The point $a$ is determined up to 2-torsion, so in the smooth case, the map $f$ is $4:1$, in the singular case the map is $2:1$ and $1:1$ if the curve is a node or cusp respectively. This means there exists four, two, or one suitable triples $y_1,y_2,y_3$ respectively. Still, we need to show that at least one is real and satisfying $y_i\neq p_j'$. 

In the elliptic quartic case, the only case where we have no real solutions for (\ref{eq_gruppevirkning}) is if $C'$ consists of two connected components and the point $p_3'-p_1'-p_2'$ lies on the component not containing the origin (only in this case do no real tangents pass through $p_3'-p_1'-p_2'$). Since $p_3'-p_1'-p_2'$ lies on the component of $C'$ that contains an odd number of cameras, we can ensure that there is a real solution by taking our projection center $x$ to lie on whichever component of $C$ that has an odd number of camera centers.

In the nodal case, it might also happen that neither of the two solutions is real. Like in the smooth case, this can be avoided by choosing the initial projection center on the quartic curve to lie on the side of the singularity that contains an odd number of cameras. In the cuspidal case, the point $a$ is uniquely determined and necessarily real since complex solutions would appear as conjugate pairs.

This ensures that there exists a real solution. Next, given a solution to (\ref{eq_gruppevirkning}), we take
\begin{align*}
y_1&=p_2'-a,\\
y_2&=p_1'-a,\\
y_3&=a,
\end{align*}
so the only way $y_i= p_j'$ is a solution is if $3p_j'-p_i'+p_k'=0$. $y_i=p_i'$ is only a solution if $p_i'+p_j'+p_k'=0$. There is a finite number of choices of $x$ where this is the case, so we can again avoid the issue by picking an appropriate $x$.

In all these cases, the curve $C$ (a quadric passing through all camera centers) is of type $(2,2,1,1)$, so by \cref{prop:bidegree_on_P_->_bidegree_on_Q}, its conjugate is also of type $(2,2,1,1)$.
\end{proof}

\begin{remark}
The proof above shows that nodal or cuspidal quartic curves are critical by constructing a conjugate configuration that is of the same type, nodal or cuspidal curves respectively. The nodal and cuspidal quartic curves also have another family of conjugates (the elliptic ones do not). Unlike the elliptic quartic case, the singular curves are rational, this means they appear as critical configurations when the three quadrics $S_P^{ij}$ coincide as a cone (they are of type (1,3,1,1) in \cref{tab:curves_on_same_quadric}). This is proven in \cref{prop:twisted_cubic_and_secant_line}, the conjugates are twisted cubics with a secant or tangent line respectively.
\end{remark}

In the case where the elliptic quartic only passes through two of the camera centers, the curve is of type $(2,2,1,0)$ on two of the quadrics and $(2,2,1,1)$ on the last. If such a configuration was critical, the conjugate curve would have to be of type $(2,3,2,1)$ on two of the quadrics and $(2,2,1,1)$ on the last, but a curve of degree 5 does not appear as the intersection of quadrics, so this can not be a critical configuration. A similar problem appears if the curve is singular and one of the camera centers lies on the singular point 

\subsubsection{The twisted cubic}
\label{sec:cubic_curves}
Having covered the quartic curves, we move one degree lower, to the cubics. Plane cubics do not appear as the intersection of quadrics, so the only irreducible cubic curve that might be critical is the twisted cubic curve. Unlike the case with the elliptic quartic, we no longer require all camera centers to lie on the curve. We consider the possible configurations of cameras one by one.

\begin{proposition}
\label{prop:twisted_cubic_is_critical}
Any configuration consisting of three cameras along with a set of points all lying on the strict transform of a twisted cubic passing through the camera centers is critical. The conjugate consists of points on the strict transform of a conic curve not passing through the camera centers. 
\end{proposition}

\begin{proof}
Given a twisted cubic, there is a three-dimensional family of quadrics containing it. This leaves us 2 degrees of freedom for each of the quadrics. In other words, for each line intersecting the curve, there is a unique quadric containing both the cubic and the line (unless the line is a secant, in which case a whole pencil of quadrics contains it). 

Take a line $L$, secant to $C$, and not passing through the camera centers. There exists a pencil of quadrics containing both $C$ and $L$. By \cref{lem:compatible_line}, any three distinct quadrics in this pencil form a compatible triple of quadrics. The line $L$ has bidegree $(1,0)$ and is a secant to $C$, so it intersects $C$ twice. It follows that the twisted cubic is of type $(1,2,1,1)$, so the conjugate is of type $(1,1,0,0)$, a conic not passing through the camera centers.
\end{proof}

\begin{proposition}
Any configuration consisting of three cameras along with a set of points lying on the strict transform of a the union of a twisted cubic and a line secant to the cubic is critical if all the camera centers lie on the union of the cubic and secant line. The conjugate is a curve of the same type.
\end{proposition}

\begin{proof}
See \cite{anyViews}, Propositions 7.15 and 7.16.
\end{proof}

If we take a twisted cubic passing through only one camera center, and the remaining two not on a secant, we end up with only 1 degree of freedom for two of the quadrics and 0 for the last. Moreover, a twisted cubic not passing through any camera center fixes all three quadrics. With so few degrees of freedom, one can typically not find a compatible triple of quadrics, so neither of these configurations is critical$-$that is, unless the three camera centers happen to be collinear. To prove this, we first need the following (seemingly unrelated) lemma:

\begin{lemma}
\label{lem:cubic_and_line}
Let $C\in\p3$ be a twisted cubic, let $l$ be a line in $\p3$, not secant to $C$, and let $p$ be a fixed point on $l$ not lying on the twisted cubic. There exists a point $x\in C$ such that for every point $q\in l$, such that $q\neq p$ and $q\notin C$, the unique quadric surface containing $C$, $p$, and $q$, also contains the line $\overline{xq}$.
\end{lemma}

\begin{proof}
Let $s$ denote the unique secant line of $C$ passing through $p$. The plane spanned by $s$ and $l$ intersects $C$ in 3 points. Two of these points lie on $s$, the last one does not. We denote this one by $x$. 

In the pencil of quadrics containing $C$ and $p$, each quadric also contains the secant $s$. For any point $q$ as above, the quadric containing $C$, $p$, and $q$, intersects the line $\overline{xq}$ in three points, $x$, $q$, and a third point lying on $s$. Since it contains three points on $\overline{xq}$, it contains the whole line. 
\end{proof}

\begin{proposition}
Any configuration consisting of three collinear cameras along with a set of points lying on the strict transform of a twisted cubic not intersecting the line with the camera centers is critical. The conjugate is a rational quartic curve containing the three camera centers.
\end{proposition}

\begin{proof}
Let $C$ denote the twisted cubic, let $p_i$ denote the camera centers, and let $L$ denote the line spanned by the camera centers.

Each quadric needs to contain $C$, as well as two camera centers, which means all three quadrics are fixed. However, we show that even without any freedom to choose the quadrics, the triple is still compatible.

Let $S_P^{ij}$ be the unique quadric containing $C$, $p_i$, and $p_j$. If $S_P^{ij}$ is smooth there are two lines through $p_i$. One is a secant to $C$, the other line intersects $C$ exactly once; denote the latter line by $g_{P_i}^{j}$. If $S_P^{ij}$ is a cone, there is only one line through $p_i$, in this case, we take this one to be $\gp{i}{j}$. Repeating this process for each pair of cameras on each of the three quadrics, we get $6$ lines, two on each quadric, and two through each camera center. 

By \cref{lem:cubic_and_line}, the lines $g_{P_i}^{k}$ and $g_{P_j}^{k}$ intersect $C$ at the same point, which we denote by $x_k$ (see \cref{fig_cubicline}). For now, assume that the two lines $g_{p_i}^{j}$ and $g_{p_i}^{k}$ through each $p_i$ are distinct, and denote their spanned plane by $\Pi_i$. 

The plane $\Pi_{1}$ intersects $S_P^{12}$ in two lines. The first line is $g_{p_1}^{2}$, whereas the second is the unique line on $S_P^{12}$ that is secant to $C$ \emph{and} passes through $x_3$. Denote the second line by $l_{12}$. The plane $\Pi_{2}$ intersects $S_P^{12}$ in two lines as well, the first being $g_{p_2}^{1}$, and the latter being $l_{12}$. Repeating this construction with the two other quadrics, we get two other lines $l_{13}$ and $l_{23}$. The three lines $l_{ij}$ along with their point of intersection $x$ satisfy the conditions for \cref{lem:compatible}, so the quadrics $S_P^{12},S_P^{13},S_P^{23}$ do indeed form a compatible triple.

\begin{figure}
\begin{center}
\includegraphics[width = 0.40\textwidth]{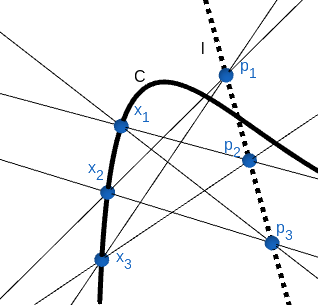}
\end{center}
\caption{The twisted cubic $C$ and the line $l$. The line $g_{p_i}^{j}$ is the line spanned by $p_i$ and $x_j$.}
\label{fig_cubicline}
\end{figure}

Since the twisted cubic is of type $(1,2,0,0)$, its conjugate is of type $(1,3,1,1)$, a rational quartic curve passing through all three camera centers. This is one of the curves from \cref{prop:rational_curves}, appearing when all the quadrics $S_Q^{ij}$ on the other side coincide.

In the construction above, the pair of lines $g_{P_i}^{j}$ through each camera center may coincide. In this case, the three points $x_i$ also coincide, and the three quadrics $S_P^{ij}$ are all cones. By \cref{prop:compatible_collinear} the quadrics are compatible in this case also, and the camera centers on the rational quartic on the other side are collinear.
\end{proof}

In the case where the line containing the cameras intersects the twisted cubic, but is not a secant or tangent, the three quadrics $S_P^{ij}$ are all equal. In this case, the only way the lines $g_{P_i}^{j}$ can be permissible is if they are all secant to the twisted cubic, but in this case, the twisted cubic is of type $(2,1,0,0)$ (or type $(3,1,1,1)$ if you include the line) and neither of these appear as critical configurations when the three quadrics intersect. Hence no such critical configurations exist.

\begin{proposition}
\label{prop:twisted_cubic_and_secant_line}
Any configuration consisting of three collinear cameras along with a set of points lying on the strict transform of a twisted cubic having the line with the camera centers as a secant or tangent is critical. The conjugate is the strict transform of a singular rational quartic curve passing through the three camera centers, a nodal curve if the line with the cameras is secant, and a cuspidal curve if the line is tangent.
\end{proposition}

\begin{proof}
Let $C$ denote the twisted cubic, and let $L$ denote the line spanned by the camera centers. Since the line $L$ containing the cameras is a secant (or tangent) to $C$, any quadric containing $C$ and two cameras contains all of $L$. The union of a twisted cubic and a secant (or tangent) line is contained in a pencil of quadrics, leaving us 1 degree of freedom when choosing $S_P^{ij}$. In particular, for each line secant to $C\cup L$, there is a unique quadric containing $C\cup L$ and the secant.

Let $x\in C$ be a point not lying on $L$, chosen in the following way:
\begin{itemize}
\item If $L$ intersects $C$ in two real, distinct points, these two points cut both $C$ and $L$ into two connected parts, one that intersects the plane at infinity, and one that does not. One part contains an even number of cameras, let $x$ lie on the corresponding part of $C$.
\item If $L$ is a tangent to $C$, or if the two intersection points are not real, let $x$ be any point on $C$ not lying on $L$.
\end{itemize}
Projecting from $x$, we get a conic curve $C'$ in the plane, along with a line $L'$ containing the projections $p_1',p_2',p_3'$ of the camera centers.

\begin{figure}
\begin{center}
\includegraphics[width = 1.0\textwidth]{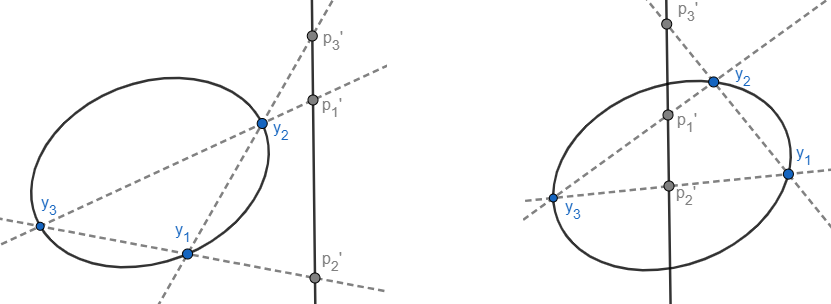}
\end{center}
\caption{By our choice of $x$, zero (left) or two (right) camera centers will lie \enquote{inside} the conic, in both cases there exists a triangle $y_1,y_2,y_3$ with vertices on the conic such that $p_i'$ lies on $\overline{y_jy_k}$.}
\label{fig_cubic+line}
\end{figure}

We can choose coordinates in such a way that $C'$ is an ellipse. By choosing $x$ as above, we have ensured that among $p_1',p_2',p_3'$, an even number lie \enquote{inside} $C'$. Without loss of generality, we can assume that $p_3'$ lies \enquote{outside} $C'$ (see \cref{fig_cubic+line}). Let $y_3$ be a point on $C'$, the line spanned by $y_3$ and $p_1'$ intersects $C'$ in a second point, we denote this by $y_2$. Similarly, we denote by $y_1$ the second point where the line spanned by $y_3$ and $p_1'$ intersects $C'$. As $y_3$ moves along the conic, the line $\overline{y_1y_2}$ hits every point on $L'$ \enquote{outside} $C'$ twice, fix $y_3$ to be one of the two points where $\overline{y_1y_2}$ passes through $p_3'$(see \cref{fig_cubic+line}).

The preimages of the three points $y_i$ are three lines through $x$. Denote the preimage of $y_i$ by $l_{jk}$, and let $S_P^{ij}$ be the unique quadric containing $C\cup L$ and the line $l_{ij}$. Then the point $x$, along with the three lines $l_{ij}$ satisfies the conditions in \cref{lem:compatible}, so the three quadrics $S_P^{12},S_P^{13},S_P^{23}$ constitute a compatible triple. Hence the configuration is critical.

The three quadrics $S_Q^{ij}$ on the other side are all equal to the same cone. The twisted cubic is of type $(1,2,0,0)$, so its conjugate is of type $(1,3,1,1)$, a rational quartic curve with a singularity in the vertex of the cone. These are the curves discussed in \cref{prop:elliptic_curve_is_critical}. Note that the line $L$ is also part of the critical locus, with its conjugate being the vertex of the cone on the other side.
\end{proof}

\subsubsection{Curves of lower degree}
\label{sec:curves_of_lower_degree}
We are now left with curves appearing as the intersection of quadrics where all components are of degree at most 2. Most of these are subsets of the \enquote{plane + conic} configuration, and hence critical, having conjugates that are of the same type (see \cref{prop:plane+conic_critical}). There are also some that do not appear as subsets of the \enquote{plane + conic} configuration, these are degeneration of the cases already discussed above and can be proven to be critical using proofs similar to the non-degenerate cases, these are:
\begin{itemize}
\item Union of two conics containing the camera centers, not all camera centers on one conic (degenerate elliptic quartic)
\item Union of conic and two intersecting lines containing the camera centers, one or two cameras lying on the conic (degenerate elliptic quartic)
\item Union of four lines containing the camera centers, each line intersecting exactly two others (degenerate elliptic quartic)
\item Union of conic and a line intersecting the conic once, camera centers either collinear or at least two lying in the union (degenerate twisted cubic)
\item Union of three lines, one line intersecting the other two, camera centers either collinear or at least two lying on the lines (degenerate twisted cubic)
\end{itemize}

Moreover, among the configurations that \emph{are} subsets of the \enquote{plane + conic} configuration, some also appear as the intersection of non-reducible quadrics, these have other conjugates that are not of the same type. For the sake of completeness, we give a summary of them here. No proof is given, but the reader can verify that the curves lie on the intersection of three compatible irreducible quadrics. This task should be fairly straightforward since there are many degrees of freedom for the quadrics when the curves are of a low degree.
\begin{enumerate}
\item Two intersecting lines on one side, and two disjoint lines on the other. The number of cameras lying on the two lines is equal on both sides.
\item Cameras lying on a conic curve. The conjugate is a line not passing through the camera centers.
\item Two lines not passing through the camera centers. This one has two conjugates, the first is three cameras lying on the union of 3 lines (one intersecting the other two). The second conjugate consists of points lying on the union of a conic and a line, and the cameras lying on the conic. This triple of conjugates is a degeneration of the \enquote{twisted cubic $\leftrightarrow$ conic} configuration shown in \cref{prop:twisted_cubic_is_critical}.
\end{enumerate}

\subsection{Quadrics intersecting in a finite number of points}
\label{sec:quadrics_intersecting_in_points}
Lastly, we have the case where three quadrics intersect in a finite number of points. This is the \enquote{general} case in the sense that with two general triples of cameras, the set of critical points on each side is finite. Indeed, on each side, we get a compatible triple of quadrics, whose intersection is 8 points (counted with multiplicity and including possible complex points), 7 of these points lie in the set of critical points, whereas the last one does not (recall \cref{lem:remove_line_point_or_plane}).

Conversely: given a set of seven points $X$ and three cameras $P_i$, when is $(P_1,P_2,P_3,X)$ a critical configuration? Since 9 points are enough to fix a quadric, there is generally only one triple of quadrics $S_P^{ij}$ such that $p_i,p_j\in S_P^{ij}$ and $X\in\widetilde{S_P^{ij}}$ where $\widetilde{S_P^{ij}}$ is the strict transform of $S_P^{ij}$ under the blow-up of $p_i,p_j$ and the total transform under the final blow-up (in other words, $\widetilde{S_P^{ij}}$ belongs to the class $2H-E_i-E_j$). Note that this is generally the case even if one (or more) of the points in $X$ happen to lie on an exceptional divisor over a camera center, say $p_1$. In this case, the quadric $S_P^{12}$ is fixed by the constraints of having to contain $\pi(X)$, $p_1$, $p_2$ as well as having a certain line through $p_1$ as a tangent, likewise for $S_P^{13}$, the quadric $S_P^{23}$ is fixed to be the one containing $\pi(X)$ and all three camera centers.

So generally the triple of quadrics will be fixed. Now IF this triple of quadrics is compatible, it comes with six permissible lines, spanning three planes. The three planes intersect in one of the 8 points lying on the intersection of the three quadrics. The configuration $(P_1,P_2,P_3,X)$ is critical if and only if the triple of quadrics is compatible \emph{and} the three planes do not intersect in one of the points in $X$ (by \cref{thr:compatible_implies_critical}). This gives us the compatibility condition for 7 general points.

There exist certain positions of points and camera centers such that they do not fix all three quadrics. In this case, the configuration is still critical as long as there exists a suitable triple of quadrics, the points and camera centers do not need to fix the quadrics. The conditions are summarized in \cref{prop:7_points_critical}.

Note that if $S_P^{ij}$ is not fixed, this is because there exists a (possibly degenerate or reducible) elliptic quartic curve that passes through $p_i,p_j$ and whose strict transform contains all of $X$. In the previous section we have shown that if the curve also passes through the final camera, the whole curve is part of a critical configuration, so the seven points along with the cameras will be a critical (but not maximal) configuration in this case.

\begin{proposition}
\label{prop:7_points_critical}
Let $P_1,P_2,P_3$ be three cameras, and let $X\subset\Bl_{\textbf{P}}(\p3)$ be a set consisting of seven points. Then $(P_1,P_2,P_3,X)$ is critical if and only if there exists three quadrics $S_P^{ij}$ satisfying the following conditions:
\begin{enumerate}
\item $S_P^{ij}$ contains the two camera centers $p_i,p_j$.
\item $X\subset\widetilde{S_P^{12}}\cap \widetilde{S_P^{13}}\cap\widetilde{S_P^{23}}$
\item $S_P^{12},S_P^{13},S_P^{23}$ is compatible.
\item The intersection of the three planes spanned by the six permissible lines $\gp{i}{j}$ does not lie in $\pi(X)$.
\end{enumerate}
\end{proposition}
\begin{remark}
A sextuple of permissible lines $\gp{i}{j}$ only exists if the quadrics are compatible, so the first two conditions need to be fulfilled for the third to make sense. 
\end{remark}

For the sake of completeness, we examine what happens if $X$ does not consist of exactly 7 points. For 8 or more points, a configuration is critical if and only if it is a subconfiguration of one of the critical configurations mentioned earlier in the section, i.e. if the points lie on a critical curve or surface. For 6 or fewer points, the configuration is always critical \cite{6points}, but not maximal, since any configuration is contained in a critical configuration consisting of at least 7 points. In particular, the generic 6-point configuration is contained in exactly 3 critical 7-point configurations.

The critical configurations for three views are summarized in \vref{fig_12critical}. For illustration purposes, the blow-up is not shown, the actual configurations consist of points lying on the strict transform of the varieties depicted in \cref{fig_12critical}.

\section*{Acknowledgments}
I would like to thank my two supervisors, Kristian Ranestad and Kathlén Kohn, for their help and guidance, for providing me with useful insights, and for their belief in my work. I thank the anonymous referee for a very thorough and helpful report. I would also like to thank the Norwegian National Security Authority for funding my project. 

\addcontentsline{toc}{section}{References}

\printbibliography

\end{document}